\documentclass[10pt]{amsart}

\usepackage[symbol]{footmisc}
\usepackage{amsthm}
\usepackage{mathtools}
\usepackage{scalerel}
\usepackage{amsfonts}
\usepackage{amssymb}
\usepackage{mathrsfs}
\usepackage{ifsym}
\usepackage{dirtytalk}
\usepackage{stmaryrd}

\usepackage{xcolor}
\definecolor{done}{cmyk}{0.4,0,0,0}
\definecolor{todo}{cmyk}{0,0.4,0,0}
\definecolor{thilo}{cmyk}{0.1,0,0.8,0}
\definecolor{AC}{cmyk}{0.1,0,0,0.1}

\usepackage[backgroundcolor = todo]{todonotes}
\usepackage{enumitem} %added by Alberto

\usepackage{microtype} % Added by Thilo
\usepackage[unicode]{hyperref}% Added by Thilo
\usepackage{doi}
\usepackage{zref-clever} % Added by Thilo
\zcsetup{countertype={subsection=subsection}}
\zcRefTypeSetup{subsection}{Name-sg=Subsection}

\zcRefTypeSetup{theo}{Name-sg=Theorem, name-sg=theorem, Name-pl=Theorems, name-pl=theorems}
\zcRefTypeSetup{defi}{Name-sg=Definition, name-sg=definintion, Name-pl=Definitions, name-pl=definitions}
\zcRefTypeSetup{lemm}{Name-sg=Lemma, name-sg=lemma, Name-pl=Lemmata, name-pl=lemmata}
\zcRefTypeSetup{prop}{Name-sg=Proposition, name-sg=proposition, Name-pl=Propositions, name-pl=propositions}
\zcRefTypeSetup{coro}{Name-sg=Corollary, name-sg=corollary, Name-pl=Corollaries, name-pl=corollaries}
\zcRefTypeSetup{prob}{Name-sg=Problem, name-sg=problem, Name-pl=Problems, name-pl=problems}
\zcRefTypeSetup{conj}{Name-sg=Conjecture, name-sg=conjecture, Name-pl=Conjectures, name-pl=conjectures}
\zcRefTypeSetup{ques}{Name-sg=Question, name-sg=question, Name-pl=Questions, name-pl=questions}
\zcRefTypeSetup{examples}{Name-sg=Examples, name-sg=examples, Name-pl=Examples, name-pl=examples}
\zcRefTypeSetup{remark}{Name-sg=Remark, name-sg=remark, Name-pl=Remarks, name-pl=remarks}

\AddToHook{env/defi/begin}{\zcsetup{countertype={theo=defi}}}
\AddToHook{env/lemm/begin}{\zcsetup{countertype={theo=lemm}}}
\AddToHook{env/prop/begin}{\zcsetup{countertype={theo=prop}}}
\AddToHook{env/coro/begin}{\zcsetup{countertype={theo=coro}}}
\AddToHook{env/prob/begin}{\zcsetup{countertype={theo=prob}}}
\AddToHook{env/conj/begin}{\zcsetup{countertype={theo=conj}}}
\AddToHook{env/ques/begin}{\zcsetup{countertype={theo=ques}}}
\AddToHook{env/remark/begin}{\zcsetup{countertype={theo=remark}}}
\AddToHook{env/examples/begin}{\zcsetup{countertype={theo=examples}}}

\newtheorem{theo}{Theorem}[section]
\newtheorem{lemm}[theo]{Lemma}
\newtheorem{prop}[theo]{Proposition} %Added by Thilo
\newtheorem{coro}[theo]{Corollary} %Added by Thilo
\newtheorem{conj}[theo]{Conjecture} %Added by Thilo
\newtheorem{ques}[theo]{Question} %Added by Thilo

\theoremstyle{definition}
\newtheorem{defi}[theo]{Definition}
\newtheorem{examples}[theo]{Examples} %Added by Garrett
\newtheorem{prob}[theo]{Problem}

\theoremstyle{remark}
\newtheorem{remark}[theo]{Remark}

\newcommand{\cref}[1]{\zcref{#1}}
\newcommand{\Cref}[1]{\zcref[S]{#1}}

\numberwithin{equation}{section}

\newcommand{\un}[1]{\mathrel{\not\mspace{-5mu}#1}}
\newcommand{\isntcontainedin}{\not\leqslant}
\newcommand{\equim}{\equiv}

\newcommand{\otp}{otp}
\newcommand{\ZZ}{\mathbb{Z}}
\newcommand{\QQ}{\mathbb{Q}}
\newcommand{\RR}{\mathbb{R}}
\newcommand{\NN}{\mathbb{N}}

\newcommand{\card}[1]{|#1|}
\newcommand{\otR}{\lambda}
\newcommand{\AC}{\ensuremath{\mathsf{AC}}}
\newcommand{\DC}{\ensuremath{\mathsf{DC}}}
\newcommand{\AD}{\ensuremath{\mathsf{AD}}}
\newcommand{\BE}{\ensuremath{\mathsf{BE}}}
\newcommand{\ZF}{\ensuremath{\mathsf{ZF}}}
\newcommand{\ZFC}{\ensuremath{\mathsf{ZFC}}}
\newcommand{\PFA}{\ensuremath{\mathsf{PFA}}}
\newcommand{\OP}{\mathsf{O}}

\newcommand{\pwi}[2]{{#1}{\left[#2\right]}}
\newcommand{\cf}{\mathrm{cf}}
\newcommand{\conc}{{}^\smallfrown}
\newcommand{\soast}[2]{
\mathord {\left\{ #1 \middle| #2\right\}}
}

\newcommand{\soafft}[2]{{}^{#1} #2}

\newcommand{\opin}[2]{{
\mathord ( #1, #2)
}}

\newcommand{\ou}{ou}
\newcommand{\ue}{ue}

\newcommand{\z}{s}
\newcommand{\en}{en}

\newcommand{\seq}[2]{\langle #1 \mid #2\rangle}
\newcommand{\opair}[2]{{
\mathord \langle #1, #2\rangle
}}

\title{Untranscendable order types}
\author{Garrett Ervin} 
\email{garrette@ttk.elte.hu}
\author{Alberto Marcone}
\address{Dipartimento di scienze matematiche, informatiche e fisiche, Universit\`a di Udine, Via delle Scienze 208, 33100 Udine --- Italy}
\email{alberto.marcone@uniud.it}
\author{Thilo Weinert}
\email{thilo.weinert@univie.ac.at}


\thanks{Marcone and Weinert were supported by the Italian PRIN 2022 ``Models, sets and classifications'', prot.\ 2022TECZJA. Marcone is a member of INdAM-GNSAGA. The authors wish to thank Rapha\"el Carroy, Elliot Glazer, Asaf Karagila, and Benedikt Loewe for enlightening conversations.}

\begin{document}

\begin{abstract}
We introduce and study a multiplicative analog\ue\ of additive indecomposability for linear order types that we call \emph{untranscendability}, as well as a strengthening that we call \emph{$s$-untranscendability}. We show that, with the unique exception of the two-point type, every untranscendable type is additively indecomposable, and every $\sigma$-scattered untranscendable type is strongly indecomposable. Under the Proper Forcing Axiom, every untranscendable Aronszajn type is strongly indecomposable. We also show that a theorem of Hagendorf and Jullien, that every scattered indecomposable type must be strictly indecomposable to either the left or right, has a natural multiplicative analog\ue\ for $s$-untranscendable types. 
\end{abstract}

\maketitle

\section{Introduction}
A linear order $X$ is \emph{additively indecomposable} (or simply \emph{indecomposable}) if whenever $X$ is expressed as an ordered sum $X = A + B$, then either $X$ embeds in $A$ or $X$ embeds in $B$. It is \emph{strongly indecomposable} if whenever $X$ is expressed as a union of two suborders $X = A \cup B$, then either $X$ embeds in $A$ or $X$ embeds in $B$. Indecomposable orders arise naturally in the general study of linear orders, and they play a central role in the structure theory of several canonical classes of linear orders in particular.

\begin{itemize}[leftmargin=*,label=$\diamond$]
    \item For the class of ordinals, Cantor showed that the nonzero indecomposable ordinals are precisely those of the form $\omega^{\beta}$, and that every ordinal can be expressed as a finite sum of indecomposable ordinals.
    \item Generali\z ing Cantor's result, Laver showed in \cite{971L0} that every scattered linear order can be expressed as a finite sum of indecomposable scattered orders.
    \item Whereas for ordinals the notions of indecomposability and strong indecomposability coincide, the same is not quite true for scattered linear orders. Laver clarified the distinction precisely by showing that the class of indecomposable scattered order types is formed by closing the class $\{0, 1\}$, consisting of the empty type and singleton type, under so-called \emph{regular unbounded sums}, whereas the class of strongly indecomposable scattered order types is formed by closing $\{0, 1\}$ under \emph{regular increasing sums} and equimorphism (see \cite[Theorems 2.4 and 2.13]{973L1}). 
  \item Laver then generali\z ed these results to characteri\z e the indecomposable and strongly indecomposable $\sigma$-scattered linear orders. These are obtained by further closing $\{0,1\}$ under regular unbounded shuffles and regular increasing shuffles, respectively (see \cite[Theorem 3.2]{973L1}).
    \item Preceding Laver's analysis, Jullien \cite{968J0} studied the class of scattered orders from an abstract structural perspective, and proved his \emph{indecomposability theorem} which states that every indecomposable scattered order is strictly indecomposable to either the left or right. 
    \item Hagendorf generali\z ed Jullien's theorem to all linear orders, showing in \cite{977H0} that an arbitrary indecomposable linear order $X$ is either strictly indecomposable to one side, or else embeds a copy of its double $X + X$. We will call this result the \emph{Hagendorf-Jullien theorem}.
    \item More recently, it has been shown that in the presence of strong forcing axioms the class of Aronszajn lines has a structure theory that resembles the structure theory of the $\sigma$-scattered linear orders. In particular, Barbosa \cite{023B0} proved that under the Proper Forcing Axiom (\PFA), every Aronszajn line can be written as a finite sum of indecomposable Aronszajn lines. 
\end{itemize}

Less is known about products of linear orders than about their sums, and as far as we know no general notion of multiplicative indecomposability for linear orders has been studied previously.

With a few exceptions (\cite{952D0}, \cite{952D1}, \cite{952D2}, \cite{953G0},
\cite{954G0}, \cite{955G0}, \cite{959M0}, \cite{960CM0}, \cite{952DS0}) most of the research into the behavior of multiplication for linear order types has focused on ordinals. A central goal of \en quiry was to study transfinite analog\ue s of the fundamental theorem of arithmetic (cf.\ \cite{001AO0}).

For ordinals in particular, three multiplicative properties have been investigated that are analogous to additive indecomposability for ordinals in one or another of its forms.

An ordinal $\alpha$ is said to be \emph{multiplicatively principal} if $0 < \beta < \alpha$ implies $\beta\alpha = \alpha$ (cf.\ \cite[Chapter XVIII]{906H0}, \cite{909J0}) and it is a \emph{$\delta$-number} if it is an indecomposable $\omega^\beta$ whose exponent $\beta$ is indecomposable itself (cf.\ \cite[Chapter XXI, §82, page 602, LXIX]{906H0}).  Furthermore, $\alpha$ is called \emph{multiplicatively decomposable}\label{multiplicative_decomposability} if there are ordinals $\beta, \gamma < \alpha$ such that $\beta\gamma = \alpha$ and \emph{multiplicatively indecomposable} otherwise (cf.\ \cite[§19]{966B0}), the latter being the notion of interest here.

It turns out that an ordinal $\alpha$ is multiplicatively principal if and only if either $\alpha < 3$ or $\alpha$ is a $\delta$-number. Furthermore it is multiplicatively decomposable unless it is multiplicatively principal, a natural prime number, or $\beta + 1$ for an indecomposable $\beta$ ($2$ belonging to all three families). See furthermore \cite{909J1}, \cite{946S0}, \cite{950C0}, \cite{951A0}, \cite{951N0}, \cite{982J0}, and \autoref{untrans_ord}.

In this paper, we introduce and study multiplicative analog\ue s of additive indecomposability generalizing the multiplicative notions discussed above for ordinals to all linear orders. We say that a linear order $X$ is \emph{untranscendable} if whenever $A$ and $B$ are suborders of $X$ and $X$ embeds in their product $AB$, then either $X$ embeds in $A$ or $X$ embeds in $B$ (see \Cref{untrdef}); $X$ is \emph{$s$-untranscendable} if this happens even when $A$ and $B$ are not assumed to be suborders of $X$ (\Cref{s-untrdef}). 

We will show that many of the basic structural facts about indecomposable orders have multiplicative versions that hold for untranscendable orders. Our main result in this direction is \Cref{Garrett's Theorem}, which is an analog\ue\ of the Hagendorf-Jullien theorem for $s$-untranscendable orders. It says roughly that an $s$-untranscendable order type either is strictly untranscendable to one side, or else embeds a copy of its square. 

We also study the relationship between untranscendability, indecomposability, and strong indecomposability. In \Cref{AD+UT=2} we show that, with the unique exception of the two-point order type $2$, every untranscendable type is indecomposable. 

On the other hand, untranscendability does not imply strong indecomposability in general. Again, $2$ is a counterexample. A construction of Sierpi\'nski implies that the order of the real numbers $\mathbb{R}$ is also a counterexample (see the discussion following \Cref{stronglyindecdef}). This answers \cite[Question F]{020KW0} negatively. The construction, however, requires the Axiom of Choice. 

In \Cref{OpenProbs} we mention some open problems, including whether in the absence of choice it may be that, except for $2$, untranscendability implies strong indecomposability.
We show in \Cref{2 only FF} that untranscendability implies strong indecomposability outright (always with the exception of $2$) if one assumes a natural extra hypothesis. This hypothesis is satisfied by classes of orders like the $\sigma$-scattered linear orders, or the Aronszajn lines under \PFA, whose indecomposable members can be constructed inductively via regular unbounded sums and shuffles. It follows from \Cref{2 only FF} that every untranscendable $\sigma$-scattered type except $2$, as well as every untranscendable Aronszajn type under \PFA, is strongly indecomposable. See \Cref{SS+untr implies SI} and \Cref{Aron+untr implies SI}.

Along the way we give many examples of untranscendable orders and present some general methods for constructing them.\smallskip

The paper is organi\z ed as follows. 
In \Cref{Prelims} we define the basic terminology we need and state some background facts about sums and products of linear orders. 
In \Cref{IndecTypes}, we review the theory of indecomposable and strongly indecomposable orders, with an eye toward motivating our results about untranscendable orders. 

In \Cref{UntrTypes}, we define the notion of untranscendability as well as various strengthenings, including $s$-untranscendability. We establish multiplicative versions for untranscendable types of several basic facts about indecomposable types, and prove our analog\ue\ of the Hagendorf-Jullien theorem for $s$-untranscendable types. 

In \Cref{Untr+Indec}, we analyze the relationship between untranscendability and indecomposability. This analysis depends on a general method (\Cref{Embed in Y or X//Y}) for finding suborders of an untranscendable order that embed the entire order. We apply this method to prove both \Cref{AD+UT=2} and \Cref{2 only FF}, and then use these results to deduce strong indecomposability for untranscendable $\sigma$-scattered orders and Aronszajn lines under \PFA.

In \Cref{OpenProbs}, we conclude with some open problems.\smallskip 

The proofs of our main results are elementary, and the paper is largely self-contained. For general background on linear orders, the standard reference is Rosenstein's book \cite{982R0}. Our applications in the last two sections of the paper build on results of Laver, Baumgartner, Moore, Martinez-Ranero, Barbosa and others, and we also discuss untranscendability in the presence of forcing axioms, the Axiom of Choice, and the Axiom of Dependent Choice. Aside from these applications, we assume only some familiarity with standard set-theoretic notation, the notion of an ordinal, and some basic facts about ordinal arithmetic. 

\section{Preliminaries} \label{Prelims}

\subsection{Basic terminology}

A \emph{linear order} is a pair $\langle X, < \rangle$ where $X$ is a set and $<$ is an irreflexive, transitive, and total binary relation on $X$. We will often refer to a linear order $\langle X, < \rangle$ by its underlying set $X$. Given $x, y \in X$, we write $x \leq y$ to abbreviate the assertion $x < y \vee x = y$. A \emph{suborder} of $X$ is a subset $Y \subseteq X$ equipped with the inherited order from $X$. 

A \emph{right endpoint} (respectively \emph{left endpoint}) in a linear order $X$ is a point $x \in X$ that is maximal (respectively minimal) in $X$. 

A suborder $I$ of a linear order $X$ is an \emph{interval}, or \emph{segment}, or \emph{convex subset} of $X$ if whenever $x < y < z$ are points in $X$ and $x, z \in I$ then $y \in I$. Notice that singletons are intervals. We say that an interval is \emph{non-degenerate} if it contains at least two points. An interval is \emph{open} if it has neither a left nor right endpoint, and \emph{closed} if it has both a right and left endpoint.

If $Y \subseteq X$ is a suborder of $X$, the \textit{convex closure} of $Y$ is the set $\{x \in X: \textrm{$y_0 \leq x \leq y_1$ for some $y_0, y_1 \in Y$}\}$. Equivalently, the convex closure of $Y$ is the smallest (by containment) interval $I \subseteq X$ with $Y \subseteq I$.

If $I$ and $J$ are intervals in $X$, we write $I < J$ if for every $x \in I$ and $y \in J$ we have $x < y$. That is, $I < J$ if $I$ and $J$ are disjoint and $I$ lies to the left of $J$ in $X$. 

An \emph{initial segment} of $X$ is an interval $I \subseteq X$ such that whenever $y < x$ are points in $X$ and $x \in I$ then $y \in I$. Symmetrically, $J$ is a \emph{final segment} of $X$ if $X \setminus J$ is an initial segment of $X$.

Given points $x \leq y$ in $X$ we write $[x, y]$ for the interval $\{z \in X: x \leq z \leq y\}$. Likewise, the expressions $[x, y), (x, y], \opin{x}{y}$ have their usual meanings as intervals. For an arbitrary pair of points $x, y \in X$, we write $[\{x, y\}]$ to denote the interval $[x, y]$ when $x \leq y$ and $[y, x]$ when $y \leq x$. 

A linear order $X$ is \emph{complete} if whenever $I$ is a non-empty initial segment of $X$ whose corresponding final segment $J = X \setminus I$ is also non-empty, then either $I$ has a right endpoint or $J$ has a left endpoint. 

A subset $U \subseteq X$ is \textit{unbounded to the right} in $X$ (or \textit{cofinal}) if for every $x \in X$ there is $u \in U$ with $u \geq x$. It is \textit{unbounded to the left} (or \textit{coinitial}) if for every $x \in X$ there is $u \in U$ with $u \leq x$. 

Given two linear orders $X$ and $Y$, an \emph{embedding} from $X$ to $Y$ is a map $f: X \rightarrow Y$ such that $x < y$ implies $f(x) < f(y)$ for every $x, y \in X$. Embeddings are automatically injective. An embedding is an \emph{isomorphism} if it is surjective. We write $X \cong Y$ if there is an isomorphism from $X$ to $Y$, and in this case we say that $X$ and $Y$ are \emph{order-isomorphic}, or simply \emph{isomorphic}. If there are embeddings $f: X \rightarrow Y$ and $g: Y \rightarrow X$, we say that $X$ and $Y$ are \emph{equimorphic} or \emph{bi-embeddable}. Isomorphic orders are necessarily equimorphic, but equimorphic orders need not be isomorphic. 

A \emph{linear order type} $\varphi$ is an isomorphism class of linear orders. We write $\textrm{otp}\langle X \rangle$ for the order type of a linear order $X$. Two orders $X$ and $Y$ have the same order type if and only if $X$ is isomorphic to $Y$. If $\textrm{otp}\langle X \rangle = \varphi$ we say that $X$ is a \emph{representative} of $\varphi$. We will sometimes conflate an order type $\varphi$ with an order that represents it and treat $\varphi$ as a concrete linear order itself. For example, we may refer to a point $x$ in a type $\varphi$ with an expression of the form $x \in \varphi$, when really we should first fix an order $X$ of type $\varphi$ and then consider some $x \in X$.

If $\varphi$ and $\psi$ are linear order types, we write $\varphi \leqslant \psi$ if for some (equivalently, every) pair of orders $X$ and $Y$ of types $\varphi$ and $\psi$ respectively, there is an embedding $f: X \rightarrow Y$. We write $\varphi \equim \psi$ if $\varphi \leqslant \psi$ and $\psi \leqslant \varphi$, or equivalently, if every order $X$ of type $\varphi$ is equimorphic with every order $Y$ of type $\psi$. The equimorphism relation $\equim$ is an equivalence relation on the class of order types. The embeddability relation $\leqslant$ is a quasi-order on the class of order types, and a partial order on the class of equimorphism types. We write $\varphi < \psi$ to mean $\varphi \leqslant \psi$ and $\varphi \not \equim \psi$.

Given a linear order $X$, we write $X^*$ for the reverse order. The orders $X$ and $X^*$ share the same underlying set of points, but we have $x < y$ in $X^*$ if and only if $y < x$ in $X$. If $\varphi$ is the order type of $X$, we write $\varphi^*$ for the type of $X^*$. 

$\mathbb{N}$ denotes the set of natural numbers, including $0$. We write $\mathbb{Z}$, $\mathbb{Q}$, and $\mathbb{R}$ for the sets of integers, rational numbers, and real numbers respectively. We view each of these sets as equipped with their usual orders. For every $n \in \mathbb{N}$, we identify $n$ with the set $\{0, 1, \ldots, n-1\}$. Viewing $n$ as a suborder of $\mathbb{N}$, we also write $n$ for the order type of this set. We write $\omega$ for the order type of $\mathbb{N}$, and $\zeta$, $\eta$, and $\lambda$ for the order types of $\mathbb{Z}$, $\mathbb{Q}$, and $\mathbb{R}$ respectively. 

A linear order $X$ is \emph{dense} (or, for emphasis, \emph{dense as a linear order}) if it has at least two points and whenever $x < y$ are points in $X$, there is a point $z \in X$ such that $x < z < y$. 

A suborder $Y \subseteq X$ is \emph{dense in X} (or, \textit{a dense suborder of} $X$) if whenever $x < y$ are points in $X$, then there is $z \in Y$ such that $x \leq z \leq y$.

Cantor proved that every countable dense linear order without endpoints is isomorphic to $\mathbb{Q}$, cf.\ \cite[§9]{895C0}. It follows from Cantor's theorem that, up to isomorphism, $\mathbb{R}$ is the unique dense and complete linear order without endpoints that has a countable dense suborder. From these facts it follows that every open interval $I \subseteq \mathbb{Q}$ is order-isomorphic to $\mathbb{Q}$, and likewise every open interval of $\mathbb{R}$ is isomorphic to $\mathbb{R}$. 

A linear order $X$ is \emph{well-ordered} if $X$ does not embed $\omega^*$, or equivalently if every non-empty suborder of $X$ has a left endpoint. An \emph{ordinal} is a transitive set that is well-ordered by the set-membership relation $\in$. We assume familiarity with some of the basic facts and terminology concerning ordinals, including that every well-ordered set $X$ is isomorphic to a unique ordinal $\alpha$, that for ordinals $\alpha, \beta$ we have $\alpha < \beta$ if and only if $\alpha \in \beta$, and that the class of ordinals is itself well-ordered by $\in$ (and hence also by $<$). We identify each ordinal with its order type. The empty set $\emptyset = 0$ is the least ordinal. The finite ordinals are precisely the natural numbers $n$. The least infinite ordinal is $\omega$. The least uncountable ordinal is $\omega_1$. 

A \emph{cardinal} is an ordinal that is not in bijection with any smaller ordinal. A cardinal $\kappa$ is \emph{regular} if any suborder $X \subseteq \kappa$ that is unbounded to the right in $\kappa$ is of order type $\kappa$. 

A \emph{reverse ordinal} is an order of the form $\alpha^*$, for $\alpha$ an ordinal. Reverse ordinals are precisely the order types of orders that do not embed $\omega$. 

A linear order $X$ is \emph{scattered} if $X$ does not embed $\eta$. In particular, all ordinals and reverse ordinals are scattered. An order $X$ is \emph{$\sigma$-scattered} if it can be written as a countable union $X = \bigcup_{n \in \NN} X_n$ such that each of the suborders $X_n$ is scattered. In particular, all scattered orders are $\sigma$-scattered. Since singletons are scattered, all countable orders are $\sigma$-scattered. 

\subsection{Sums and products}

Suppose that $X$ is a linear order of order type $\varphi$, and for every $x \in X$ we have an order $I_x$ of type $\psi_x$. The \emph{ordered sum} over $X$ of the orders $I_x$ is the order obtained by replacing each point $x \in X$ with the corresponding order $I_x$. We denote the sum by $\sum_{x \in X} I_x$ and denote its order type by $\sum_{x \in \varphi} \psi_x$. Formally, we take $\sum_{x \in X} I_x$ to be the set of ordered pairs $\{\opair{i}{x}: x \in X, i \in I_x\}$ ordered anti-lexicographically by the rule $\opair{i}{x} < \opair{i'}{x'}$ if $x < x'$ in $X$, or $x = x'$ and $i < i'$ in $I_x$.

If there is an order $Y$ such that $I_x = Y$ for every $x \in X$, we will write $YX$ instead of $\sum_{x \in X} Y$ and call this sum the \emph{product} of $X$ and $Y$. Visually, $YX$ is the order obtained by replacing every point in $X$ with a copy of $Y$; formally, it is the anti-lexicographically ordered cartesian product $Y \times X$. If $\varphi$ and $\psi$ are the order types of $X$ and $Y$, we write $\psi \varphi$ for the order type of $YX$ and call this the product of the types $\psi$ and $\varphi$. 

If $X$ has exactly two points, so that $\varphi = 2 = \{0, 1\}$, we will write $I_0 + I_1$ for $\sum_{x \in X} I_x$ and $\psi_0 + \psi_1$ for $\sum_{x \in \varphi} \psi_x$. The order $I_0 + I_1$ is the \emph{sum} of the orders $I_0$ and $I_1$ and $\psi_0 + \psi_1$ is the sum of their types. An order $Y$ has type $\psi_0 + \psi_1$ precisely when $Y$ has an initial segment isomorphic to $I_0$ whose corresponding final segment $Y \setminus I_0$ is isomorphic to $I_1$. 

The sum and product are associative, in the sense that $(XY)Z \cong X(YZ)$ and $(X + Y) + Z \cong X + (Y + Z)$ for all orders $X, Y, Z$. We will freely drop parentheses in such expressions. Moreover observe that for an order $Y$, the product $Yn$ is isomorphic to the $n$-fold sum $Y + Y + \cdots + Y$. Likewise, for a type $\psi$, $\psi n$ and $\psi + \psi + \cdots + \psi$ coincide. We will write $Y^n$ and $\psi^n$ for the $n$-fold products $YY \cdots Y$ and $\psi\psi \cdots \psi$.

For an ordered sum $\sum_{x \in \varphi} \psi_x$, we have the identity $(\sum_{x \in \varphi} \psi_x)^* = \sum_{x \in \varphi^*} \psi_x^*$. From this we get the identities $(\varphi + \psi)^* = \psi^* + \varphi^*$ and $(\varphi\psi)^* = \varphi^*\psi^*$.

It follows from their definitions that the class of ordinals, the class of scattered orders, and the class of $\sigma$-scattered orders are each closed under taking ordered sums over members of their class. In particular, these classes are closed under products and sums. 

\subsection{Ordinal exponentiation and the Cantor normal form} \label{subsec:Cantornormalform}

We recall the definition of ordinal exponentiation and state Cantor's normal form theorem for ordinals. Later, we will use exponentiated ordinals to construct examples of untranscendable types. 

For a fixed ordinal $\alpha$, the ordinal $\alpha^{\beta}$ is defined recursively on the exponent $\beta$ ($\beta$ an ordinal):

\begin{itemize}
    \item[] $\alpha^0 = 1$,
    \item[] $\alpha^{\beta + 1} = \alpha^{\beta}\alpha$,
    \item[] $\alpha^{\beta} = \sup\{\alpha^{\gamma}: \gamma < \beta\}$, for $\beta$ a limit, 
\end{itemize}
where $\sup\{\alpha^{\gamma}: \gamma < \beta\} = \bigcup_{\gamma < \beta} \alpha^{\gamma}$.

Ordinals of the form $\omega^{\alpha}$ play a special role in the structure theory of ordinals. A nonzero ordinal is indecomposable (cf.\ \Cref{indecdefn} below) if and only if it is of the form $\omega^{\alpha}$ for some ordinal $\alpha$. 

Cantor's \emph{normal form theorem} for ordinals, cf.\ \cite[§19]{897C0}, states that every ordinal can be written as a finite sum of indecomposable ordinals. More specifically, for every nonzero ordinal $\alpha$, there is a unique finite decreasing sequence of ordinals $\alpha_n > \alpha_{n-1} > \ldots > \alpha_1 > \alpha_0$ and a corresponding sequence of nonzero natural numbers $k_n, k_{n-1}, \ldots, k_1, k_0$ such that 
\[
\alpha = \omega^{\alpha_n} k_n + \omega^{\alpha_{n-1}} k_{n-1} + \cdots + \omega^{\alpha_1} k_1 + \omega^{\alpha_0} k_0. 
\]

\subsection{Condensations}\label{HausdorffCondSection}
Given a linear order $X$, a \emph{convex equivalence relation} or \emph{condensation} of $X$ is an equivalence relation $E$ on $X$ all of whose equivalence classes are intervals. Given a condensation $E$ on $X$ and a point $x \in X$, we write $[x]_E$ (or simply $[x]$ when $E$ is understood) for the equivalence class of $x$, and $X / E$ for the set of equivalence classes. We will also call $X / E$ the condensation of $X$ by $E$. The linear order on $X$ naturally determines a linear order on $X / E$, namely the one defined by the rule $[x] < [y]$ if $[x] \neq [y]$ and $x < y$ in $X$. We call this order the \emph{induced order} on $X / E$. Intuitively, $X/E$ is the order obtained by condensing each of the intervals $[x] \subseteq X$ to a point. 

Condensations are one-to-one with ordered sums in the following sense. If $X / E$ is a condensation of a linear order $X$, then $X \cong \sum_{[x] \in X / E} \, [x]$, where in the subscript we view $[x]$ as denoting a point in $X / E$ and as a summand we view $[x]$ as denoting a linear order (specifically, an interval in $X$). In the other direction, if $X = \sum_{y \in Y} I_y$ is an ordered sum over a linear order $Y$, we may define a condensation $E$ on $X$ by the rule $x E x'$ if $x, x'$ belong to the same summand $I_y$. Then we have $X / E \cong Y$. 

A specific condensation will play an important role in what follows. Given a linear order $X$, define a binary relation $F$ on $X$ by the rule $x F x'$ if and only if $[\{x, x'\}]$ is finite. It is not hard to check that $F$ is a condensation of $X$. We call $F$ the \emph{finite condensation} of $X$.

For a given $x \in X$, we refer to the condensation class $[x]_F$ as the \emph{$F$-class} of $x$. It is also not hard to see that the order type of $[x]_F$ must either be $n$ for some natural number $n$, in which case we say that the $F$-class of $x$ is finite, or one of $\omega$, $\omega^*$, or $\zeta$. In \Cref{Untr+Indec} we will be interested in orders that contain only finitely many finite $F$-classes. 

The finite condensation is really a condensation scheme, and every fixed order $X$ carries its own finite condensation. However, we will denote this condensation by $F$, and use the terminology $F$-class and the notation $[x]_F$ for a given $F$-class, regardless of the underlying order $X$. 

\section{Indecomposable types} \label{IndecTypes}

In this section, we review some basic facts about indecomposable types, give some examples of such types, and state the Hagendorf-Jullien theorem (\Cref{Hagendorf's}). These facts will serve to motivate our results about untranscendable types in the next sections. More on indecomposable types can be found in Fra\"iss\'e's book \cite[Section 6.3]{000F0}.

\subsection{Basic facts} \label{sec:indec_basic}
We begin by recalling the definition of an indecomposable type. 

\begin{defi} \label{indecdefn}
\phantom{}
\begin{enumerate}
    \item A linear order type $\varphi$ is \emph{indecomposable} if whenever $\psi$ and $\tau$ are order types such that $\varphi = \psi + \tau$, then either $\varphi \leqslant \psi$ or $\varphi \leqslant \tau$. 
    \item A linear order $X$ is \emph{indecomposable} if its order type is indecomposable, or equivalently, if whenever $I$ is an initial segment of $X$ and $J = X \setminus I$ is the corresponding final segment, then either $X$ embeds in $I$ or $X$ embeds in $J$. 
\end{enumerate}
\end{defi}

An order type $\varphi$ is \emph{decomposable} if it is not indecomposable, that is, if there are types $\psi, \tau$ such that $\varphi = \psi + \tau$ but $\varphi \isntcontainedin \psi$ and $\varphi \isntcontainedin \tau$. 

The empty type $0$ is indecomposable, but unless otherwise noted we will usually assume when referring to an indecomposable type that it is non-empty. The singleton type $1$ is also indecomposable. The two-point type $2$ is not indecomposable, as witnessed by the decomposition $2 = 1+1$. More generally, any finite type $n>1$ is decomposable. We will give more examples of indecomposable types below. 

The following proposition says that indecomposability is an equimorphism invariant. The proof is straightforward and left to the reader. 

\begin{prop}
\label{equimorphism invariant}
Suppose that $\varphi$ and $\varphi'$ are equimorphic order types. Then $\varphi$ is indecomposable if and only if $\varphi'$ is indecomposable. 
\end{prop}

%Suppose first that $\varphi$ is indecomposable.\todo[color=green!40]{do we need to prove this in detail?} Fix orders $X$ and $X'$ of types $\varphi$ and $\varphi'$ respectively. Suppose $I'$ is an initial segment of $X'$ and $J' = X \setminus I'$ is the corresponding final segment. To show $\varphi'$ is indecomposable it suffices to show $X'$ embeds in at least one of $I'$ and $J'$. 

%Fix an embedding $f: X \rightarrow X'$. Then $I = f^{-1}[I']$ is an initial segment of $X$ with corresponding final segment $J = f^{-1}[J']$. By indecomposability, $X$ embeds in either $I$ or $J$. Since $X'$ embeds in $X$, it follows that $X'$ embeds in either $I$ or $J$. Hence $X'$ embeds in either $f[I]$ or $f[J]$, and therefore also in either $I'$ or $J'$. Thus $X'$ and its order type $\varphi'$ are indecomposable. 

%The reverse implication is symmetric. 

It will be helpful later to have in hand a slightly reformulated definition of indecomposability. We first reformulate the definition of decomposability.

\begin{prop} \label{decompequivdef}
An order type $\varphi$ is decomposable if and only if there are types $\psi < \varphi$ and $\tau < \varphi$ such that $\varphi \leqslant \psi + \tau$.
\end{prop}
\begin{proof}
If $\varphi = \psi + \tau$ is a decomposition witnessing the decomposability of $\varphi$, then $\varphi \leqslant \psi + \tau$ and $\psi, \tau < \varphi$. For the backward direction, fix orders $X, Y, Z$ of types $\varphi, \psi, \tau$ respectively, so that $X \leqslant Y + Z$ and $Y, Z < X$. Fix an embedding $f: X \rightarrow Y + Z$. Then $f^{-1}[Y]$ is an initial segment of $X$ with corresponding final segment $f^{-1}[Z]$, so that $X \cong f^{-1}[Y] + f^{-1}[Z]$. Thus if $\psi'$ and $\tau'$ are the types $f^{-1}[Y]$ and $f^{-1}[Z]$ respectively, then $\varphi = \psi' + \tau'$. And clearly $\psi' \leqslant \psi$ and $\tau' \leqslant \tau$, so that $\psi' < \varphi$ and $\tau' < \varphi$ as well.
\end{proof}

\Cref{decompequivdef} immediately yields the following. 

\begin{prop} \label{indecequivdef}
An order type $\varphi$ is indecomposable if and only if whenever $\psi$ and $\tau$ are types such that $\psi \leqslant \varphi$, $\tau \leqslant \varphi$, and $\varphi \leqslant \psi + \tau$, then either $\varphi \leqslant \psi$ or $\varphi \leqslant \tau$. \qed 
\end{prop}

We record this equivalent definition of indecomposable type since it will be directly reflected by our definition of untranscendable type, but note that there are two differences between this version of the definition and the one given in \Cref{indecdefn} above. The first is that here we are explicitly supposing that the types $\psi, \tau$ appearing in the sum are embeddable in $\varphi$. The second is that we are now considering sums $\psi + \tau$ such that $\varphi \leqslant \psi + \tau$ instead of $\varphi = \psi + \tau$. These differences do not actually change the definition, since if $\varphi \leqslant \psi + \tau$ then we can find $\psi' \leqslant \psi$ and $\tau' \leqslant \tau$ such that $\varphi = \psi' + \tau'$. 

It is easy to see that indecomposability is preserved under reversal. 

\begin{prop}
%\label{sum-lemma}
An order type $\varphi$ is indecomposable if and only if $\varphi^*$ is indecomposable. \qed
\end{prop}

\subsection{Strict and non-strict indecomposability}

It turns out that there are two essentially different kinds of indecomposable order types: types that are strictly indecomposable to one side, and non-strictly indecomposable types. Hagendorf, generali\z ing a theorem of Jullien, proved a theorem characterizing this distinction. We present Hagendorf's theorem below (\Cref{Hagendorf's}), and along the way give some examples of each kind of indecomposable type. 

\begin{prop}
\label{strong-sum-inequality-separation}
Suppose $\varphi, \psi, \tau$, and $\rho$ are order types and $\varphi + \psi \leqslant \tau + \rho$. Then either $\varphi \leqslant \tau$ or $1 + \psi \leqslant \rho$. Symmetrically, either $\varphi + 1 \leqslant \tau$ or $\psi \leqslant \rho$.
\end{prop}

\begin{proof}
In order to prove the first statement fix orders $X, Y, Z$, and $W$ of types  $\varphi, \psi, \tau$, and $\rho$, respectively, and an embedding $f: X + Y \rightarrow Z + W$. If $f[X] \subseteq Z$, then $\varphi \leqslant \tau$. If $f[X] \not\subseteq Z$ then there is an $x \in X$ such that $f(x) \in W$. Since $f$ is order-preserving we must have $f[\{x\} \cup Y] \subseteq W$, which gives $1 + \psi \leqslant \rho$. The second statement can be proved analogously.
\end{proof}

Often times, the following weak version of \Cref{strong-sum-inequality-separation} is all we need.

\begin{coro} \label{sum-inequality-separation}
Suppose $\varphi, \psi, \tau$, and $\rho$ are order types and $\varphi + \psi \leqslant \tau + \rho$. Then $\varphi \leqslant \tau$ or $\psi \leqslant \rho$.
\end{coro}

\begin{coro}
\label{containing two copies}
Suppose that $\varphi$ is a type such that $\varphi + \varphi \leqslant \varphi$. Then $\varphi$ is indecomposable. 
\end{coro}

\begin{proof}
If $\varphi = \psi + \tau$ then we have $\varphi + \varphi \leqslant \psi + \tau$ so that either $\varphi \leqslant \psi$ or $\varphi \leqslant \tau$. 
\end{proof}

From Cantor's characterizations of the order types of $\mathbb{Q}$ and $\mathbb{R}$ we get the isomorphisms $\mathbb{Q} \cong \mathbb{Q} \cap (-\infty, \sqrt{2}) \cong \mathbb{Q} \cap (\sqrt{2}, \infty)$ and $\mathbb{R} \cong \mathbb{R} \cap (-\infty, 0) \cong \mathbb{R} \cap (0, \infty)$. Hence $\eta + \eta = \eta$ and $\lambda + 1 + \lambda = \lambda$. In particular, $\eta + \eta \leqslant \eta$ and $\lambda + \lambda \leqslant \lambda$. By \Cref{containing two copies} it follows that $\eta$ and $\lambda$ are indecomposable types. 

\begin{defi}
An order type $\varphi$ is \emph{indecomposable to the right} if whenever $\varphi = \psi + \tau$ and $\tau \neq 0$ we have $\varphi \leqslant \tau$, and \emph{strictly indecomposable to the right} if in this situation we moreover have $\varphi \isntcontainedin \psi$. 

\emph{Indecomposable to the left} and \emph{strictly indecomposable to the left} are defined symmetrically. 
\end{defi}

\begin{examples} \label{indecexs} \phantom{}
\begin{enumerate}[leftmargin=*,label=(\roman*)]
    \item Since $\eta$ and $\lambda$ embed in all of their non-empty initial and final segments, these types are indecomposable to both the right and left. Neither type is strictly indecomposable to either the right or left. More generally, any type $\varphi$ embedding $\varphi + \varphi$ cannot be strictly indecomposable to one side.
    \item The type $1 + \eta$ is indecomposable since it is equimorphic to $\eta$. Since $1 + \eta$ embeds in all of its non-empty final segments, $1 + \eta$ is indecomposable to the right (but not strictly so). It is not indecomposable to the left, since its left endpoint constitutes a non-empty initial segment that does not embed $1 + \eta$.

    Since $\eta$ is indecomposable to both the left and right, this example shows that indecomposability to a given side is not invariant under equimorphism in general. However, it can be checked that strict indecomposability to one side is an equimorphism invariant.
    \item $\omega$ is strictly indecomposable to the right. More generally, it is a basic result that an ordinal is indecomposable if and only if it is of the form $\omega^{\alpha}$, and all such ordinals are strictly indecomposable to the right. 
    \item A type $\varphi$ is (strictly) indecomposable to the right if and only if $\varphi^*$ is (strictly) indecomposable to the left. For example, $\omega^*$ is strictly indecomposable to the left. 
    \item If $\varphi$ is indecomposable to the right, then $\psi \varphi$ is indecomposable to the right for any type $\psi$. More generally, if $\psi_x \leqslant \psi_y$ whenever $x < y$ in $\varphi$, it is not hard to check that the sum $\sum_{x \in \varphi} \psi_x$ is indecomposable to the right. However, this sum need not be strictly indecomposable to the right, even if $\varphi$ is strictly indecomposable to the right. 

    For example, $\omega^* \omega$ is strictly indecomposable to the right, as is the sum $\sum_{n \in \omega} (\omega^*)^n = \omega^* + (\omega^*)^2 + \cdots$. On the other hand, we get easily from Cantor's characterization of $\eta$ that $\eta \omega = \eta$. Thus $\eta \omega$ is indecomposable to the right but not strictly so, even though $\omega$ is strictly indecomposable to the right. 
    \item Symmetric comments apply on the left. For example, $\omega \omega^*$ and $\sum_{n \in \omega^*} \omega^n = \cdots + \omega^2 + \omega$ are indecomposable to the left (and strictly so). 
\end{enumerate}
\end{examples}

\Cref{Hagendorf's} below, due to Hagendorf, shows that an indecomposable type $\varphi$ is strictly indecomposable to one side precisely when $\varphi + \varphi \isntcontainedin \varphi$. Hagendorf's theorem is a generali\z ation of a theorem of Jullien \cite{968J0}, who showed that an indecomposable scattered type is always strictly indecomposable to one side.

\begin{theo}[Hagendorf, \cite{977H0}]
\label{Hagendorf's}
(The Hagendorf-Jullien theorem). If $\varphi$ is an indecomposable order type, then either $\varphi + \varphi \equiv \varphi$, or $\varphi$ is strictly indecomposable to the right, or $\varphi$ is strictly indecomposable to the left, and these three possibilities are mutually exclusive with the exception of $1$.
\end{theo}

For a proof, see \cite[6.3.4]{000F0}. The following is a reformulation of Hagendorf's theorem that we will revisit later.

\begin{coro}
\label{1+phi<=phi + phi+1<=phi  ==>  phi+phi<=phi}
If $\varphi$ is an indecomposable order type, then $\varphi \equim \varphi + \varphi$ if and only if $\varphi \equim 1 + \varphi$ and $\varphi \equim \varphi + 1$.
\end{coro}

\begin{proof}
The hypothesis $\varphi \equim 1 + \varphi \equim \varphi + 1$ gives that $\varphi$ is not strictly indecomposable to either side. 
\end{proof}

\subsection{Strong forms of indecomposability}

We introduce two strengthenings of indecomposability and give some examples. These strengthenings will play a role in \Cref{UntrTypes} and \Cref{Untr+Indec}.  

\begin{defi}
An order type $\varphi$ is \emph{sum closed} if whenever $\psi$ and $\tau$ are order types such that $\psi < \varphi$ and $\tau < \varphi$, then $\psi + \tau < \varphi$.
\end{defi} 

Observe that a sum closed type is necessarily indecomposable. The following proposition, when combined with the Hagendorf-Jullien theorem, gives that if $\varphi$ is indecomposable but not strictly indecomposable to one side, then conversely $\varphi$ is sum closed. 

\begin{prop}
If $\varphi$ is an order type such that $\varphi \equim \varphi + \varphi$, then $\varphi$ is sum closed. 
\end{prop}

\begin{proof}
Fix types $\psi, \tau < \varphi$. Then $\psi + \tau \leqslant \varphi + \varphi$, and so by hypothesis $\psi + \tau \leqslant \varphi$. If we had $\varphi \leqslant \psi + \tau$, then we would also have $\varphi + \varphi \leqslant \psi + \tau$, giving either $\varphi \leqslant \psi$ or $\varphi \leqslant \tau$ by \Cref{sum-inequality-separation} above, a contradiction. Hence $\psi + \tau < \varphi$. 
\end{proof}

Thus for non strictly indecomposable types, sum closure is equivalent to indecomposability. 

The next proposition asserts that sum closure is an equimorphism invariant. We omit the easy proof. 

\begin{prop} \label{sumclosedequiinvar}
If $\varphi$ and $\psi$ are equimorphic types, then $\varphi$ is sum closed if and only if $\psi$ is sum closed. \qed
\end{prop} 

\begin{examples} \phantom{}
\begin{enumerate}[leftmargin=*,label=(\roman*)]
    \item Since $\eta \equim \eta + \eta$ and $\lambda \equim \lambda + \lambda$, both $\eta$ and $\lambda$ are sum closed. 
    \item Types that are strictly indecomposable to one side may or may not be sum closed. For example, $\omega$ is strictly indecomposable to the right and sum closed since if $\psi, \tau < \omega$ then $\psi$ and $\tau$ are finite types.

    Using the fact that indecomposable ordinals are precisely those of the form $\omega^{\alpha}$ (cf.\ \Cref{indecexs}) and the Cantor normal form theorem (cf.\ \Cref{subsec:Cantornormalform}), it can be shown more generally that indecomposability is equivalent to sum closure for ordinals. For example, $\omega^2 = \omega \omega$ is sum closed. 

    In contrast, $\zeta \omega$ is not sum closed: while $\omega \omega < \zeta \omega$ and $\omega^* \omega < \zeta \omega$, it can be checked that $\omega \omega + \omega^* \omega \isntcontainedin \zeta \omega$. 
\end{enumerate}
\end{examples}

\begin{defi} \label{stronglyindecdef}
A linear order $X$ is \emph{strongly indecomposable} if for any partition $X = A \cup B$ of $X$ into two suborders $A$ and $B$, either $X$ embeds in $A$ or $X$ embeds in $B$. 

An order type $\varphi$ is \emph{strongly indecomposable} if some (equivalently every) linear order $X$ of type $\varphi$ is strongly indecomposable.
\end{defi}

Observe that any strongly indecomposable type is indecomposable. It is well known (and can be shown by induction on the exponent $\alpha$) that every ordinal of the form $\omega^{\alpha}$ is strongly indecomposable. Thus for ordinals (and, symmetrically, reverse ordinals) strong indecomposability is equivalent to indecomposability. In particular, $\omega$ is strongly indecomposable. It is also well known (and not hard to check) that $\eta$ is strongly indecomposable and this exemplifies something more general---in fact, all order types $\varphi$ such that $\varphi^2 \equim \varphi$ are strongly indecomposable as $\varphi^2 \equim \varphi$ means that any ordering $X$ of type $\varphi$ contains $\varphi$ copies of itself. Then any subset $Y$ of $X$ either contains one of those copies in which case $\varphi \leqslant \otp(Y)$ or is missing a point from every such copy in which case $\varphi \leqslant \otp(X \setminus Y)$. This shows that $\varphi$ being strongly indecomposable follows from the hypothesis $\varphi^2 \equim \varphi$ and this hypothesis turns out to have further implications, cf. \Cref{embedding one's square}.

While both sum closure and strong indecomposability are strengthenings of indecomposability, neither property implies the other. Sierpi\'nski, adapting an argument of Dushnik and Miller, showed that there is a partition $\mathbb{R} = A \cup B$ of $\mathbb{R}$ into two suborders $A$ and $B$ such that $\mathbb{R}$ embeds in neither $A$ nor $B$ (cf.\ \cite[Chapter 9]{982R0}; it should be noted that the argument uses the axiom of choice; see the Trichotomy Conjecture \ref{realtypeconj}). Thus $\lambda$ is not a strongly indecomposable type. But $\lambda$ is sum closed, as we observed above. On the other hand, it follows from \Cref{untranscendability of sums of finite powers} below that the sum $\Psi = \cdots + \omega^2 + \omega + 1 = \sum_{n \in \omega^*} \omega^n$ is strongly indecomposable. But it is not sum closed, since while $\omega^* < \Psi$, it is not hard to see that $\omega^* + \omega^* \isntcontainedin \Psi$. 

Before we move on the next section, recall the Hungarian arrow notation, cf.\ \cite{956ER0}. It can be employed in the context of order types: $\rho \longrightarrow (\varphi)^\psi_\kappa$ means that whenever the suborders of type $\psi$ within an order $X$ of type $\rho$ are partitioned into $\kappa$ classes, there is a suborder $Y$ of $X$ of type $\varphi$ such that all suborders of $Y$ of induced type $\psi$ are in the same class of the partition. The suborder $Y$ is then said to be \emph{homogeneous} for the partition.

It is easy to see that a partition property remains true if $\rho$ is enlarged or $\varphi$ or $\kappa$ are diminished, cf.\ \Cref{monotonicity}. If $\psi$ is finite, then the relation also remains true if $\psi$ is diminished. That this does not hold true in general follows from recent work of Gardiner, cf.\ \cite[page 4]{025G0}.

Strong indecomposability can then be expressed as a partition property: an order type is $\varphi$ is strongly indecomposable if and only if $\varphi \longrightarrow (\varphi)^1_2$.

All such partition properties are equimorphism invariants due to the following:

\begin{prop}
\label{monotonicity}
If $\kappa,\lambda$ are cardinals and $\xi,\rho,\tau,\varphi,\psi$ are order types such that $\kappa\leq\lambda$, $\rho \leqslant \xi$, $\tau\leqslant\varphi$, and $\rho \longrightarrow (\varphi)^\psi_\lambda$, then $\xi \longrightarrow (\tau)^\psi_\kappa$.
\end{prop}

\begin{proof}
Let $X$ be an order of type $\xi$ and suppose that $\chi : [X]^\psi \longrightarrow \kappa$ is a partition of $[X]^\psi$. We want to find a $\chi$-homogeneous set of order type $\tau$. To this end, fix a subset $R \subset X$ of order type $\rho$. As $\rho \longrightarrow (\varphi)^\psi_\lambda$ and $\kappa \leq \lambda$ we can find $Y \subset R$ of order type $\varphi$ which is homogeneous for $\chi \upharpoonright [R]^\psi$. As $\tau \leqslant \varphi$, there is $Z \subset Y$ of order type $\tau$ and of course this is still homogeneous for $\chi\upharpoonright [R]^\psi$ and therefore for $\chi$.
\end{proof}

\begin{coro}
 \label{strongindecequiinvar}
If $\varphi$ and $\psi$ are equimorphic types, then $\varphi$ is strongly indecomposable if and only if $\psi$ is strongly indecomposable.
\end{coro}

%\begin{proof}
%Suppose first that $\varphi$ is strongly indecomposable. Let $X$ and $Y$ be orders of types $\varphi$ and $\psi$ respectively, and suppose $Y = C \cup D$ is a partition of $Y$. Fix an embedding $f: X \rightarrow Y$. Let $A = f[X] \cap C$ and $B = f[X] \cap D$. Then $A \cup B$ is a partition of $f[X]$. Since $f[X]$ is isomorphic to $X$ and $X$ is strongly indecomposable, $X$ embeds in either $A$ or $B$. Hence $X$ embeds in either $C$ or $D$. Since $Y$ embeds in $X$, we have that $Y$ embeds in either $C$ or $D$. It follows that $Y$, and hence $\psi$, is strongly indecomposable. The argument for the backwards direction is symmetric. 
%\end{proof}

\section{Untranscendable types} \label{UntrTypes}

\subsection{Basic facts} 
We say that an order type $\varphi$ is \emph{transcendable} if there are types $\psi < \varphi$ and $\tau < \varphi$ such that $\varphi \leqslant \psi \tau$. The following is the central definition of this paper.

\begin{defi} \label{untrdef} \phantom{}
\begin{enumerate}
    \item An order type $\varphi$ is \emph{untranscendable} if it is not transcendable, that is, if whenever $\psi$ and $\tau$ are order types such that $\psi \leqslant \varphi$, $\tau \leqslant \varphi$, and $\varphi \leqslant \psi \tau$, then either $\varphi \leqslant \psi$ or $\varphi \leqslant \tau$. 
    \item A linear order $X$ is \emph{untranscendable} if $\textrm{otp}\langle X \rangle$ is untranscendable. 
\end{enumerate}
\end{defi}

The definitions of \emph{transcendable type} and \emph{untranscendable type} are obtained by replacing the sum $\psi + \tau$ with the product $\psi\tau$ in the reformulated definitions of \emph{decomposable type} and \emph{indecomposable type} given in \Cref{decompequivdef} and \Cref{indecequivdef}. In this sense, untranscendability is the multiplicative analog\ue\ of indecomposability. 

We will also consider the following strengthening of untranscendability.

\begin{defi} \label{s-untrdef}
An order type $\varphi$ is \emph{$s$-untranscendable} if whenever $\psi$ and $\tau$ are order types such that $\varphi \leqslant \psi \tau$, then either $\varphi \leqslant \psi$ or $\varphi \leqslant \tau$. 
\end{defi}

As we observed previously, an order type $\varphi$ is indecomposable if and only if whenever $\varphi \leqslant \psi + \tau$, then either $\varphi \leqslant \psi$ or $\varphi \leqslant \tau$. That is, the additive analog\ue\ of $s$-untranscendability is equivalent to indecomposability. We will see however in Example \ref{ex:untr}.\ref{ex:untr.not s-untr} that $s$-untranscendability is strictly stronger than untranscendability. The difference arises from the fact that if $\varphi \leqslant \psi + \tau$, then we can find $\psi' \leqslant \psi$ and $\tau' \leqslant \tau$ such that $\psi', \tau' \leqslant \varphi$ and $\varphi \leqslant \psi' + \tau'$. However, if $\varphi \leqslant \psi \tau$, then while we can always find $\tau' \leqslant \tau$ such that $\tau' \leqslant \varphi$ and $\varphi \leqslant \psi \tau'$, it is not true in general that we can also find $\psi' \leqslant \psi$ such that $\psi' \leqslant \varphi$ and $\varphi \leqslant \psi' \tau'$. Therefore, we could have equivalently defined $\varphi$ to be untranscendable if whenever $\psi$ and $\tau$ are order types such that $\psi \leqslant \varphi$ and $\varphi \leqslant \psi \tau$, then $\varphi \leqslant \psi$ or $\varphi \leqslant \tau$. Then $s$-untranscendability is obtained by dropping the restriction $\psi \leqslant \varphi$.

Just as sum closure is a strengthening of indecomposability, the multiplicative analog\ue\ of sum closure strengthens untranscendability. 

\begin{defi}
An order type $\varphi$ is \emph{product closed} if whenever $\psi$ and $\tau$ are order types such that $\psi < \varphi$ and $\tau < \varphi$, then $\psi \tau < \varphi$. 
\end{defi}

In Example \ref{ex:untr}.\ref{ex:untr.not s-untr} we show that product closure is a strict strengthening of untranscendability. The notion of a product closed type should also be compared with Laver's stronger notion of a \emph{regular} type; cf.\ \cite[p.\ 110]{973L1}.

It is natural to ask if either $s$-untranscendability or product closure implies the other. We will see in the next subsection that there are $s$-untranscendable types that are not product closed; in fact $\lambda$ is such a type.
On the other hand $\omega^\omega$ is product-closed and it will follow from \Cref{ut_but_not_sut_ord<=>sing_lim_ord} that it is not $s$-untranscendable.

Analog\ue s of several basic facts about indecomposable types hold for untranscendable types.

\begin{prop}
Untranscendability, product closure, and $s$-untranscendability are invariant under equimorphism.
\end{prop}

\begin{proof}
Fix equimorphic types $\varphi$ and $\varphi'$, suppose that $\varphi$ is untranscendable, and fix types $\psi', \tau' \leqslant \varphi'$ such that $\varphi' \leqslant \psi' \tau'$. Since $\varphi' \leqslant \varphi$ we have $\psi', \tau' \leqslant \varphi$, and since $\varphi \leqslant \varphi'$ we have $\varphi \leqslant \psi' \tau'$. By untranscendability of $\varphi$, we have $\varphi \leqslant \psi'$ or $\varphi \leqslant \tau'$, and hence $\varphi' \leqslant \psi'$ or $\varphi' \leqslant \tau'$, giving the untranscendability of $\varphi'$.

The arguments for $s$-untranscendability and product closure are similar and left to the reader. 
\end{proof}

\begin{prop}
Untranscendability, product closure, and $s$-untranscendability are invariant under reversal. \qed
\end{prop}

The following proposition is the multiplicative analog\ue\ of \Cref{strong-sum-inequality-separation}.

\begin{prop}
\label{new-product-lemma}
Suppose $\rho, \tau, \varphi$ and $\psi$ are order types with $\rho, \varphi$ not both $0$. If $\rho \tau \leqslant \varphi \psi$ then either $1+\rho \leqslant \varphi$ and $\rho+1 \leqslant \varphi$ or $\tau \leqslant \psi$.
\end{prop}
\begin{proof}
If $\rho=0$ then, since $\varphi \neq 0$, $1+\rho \leqslant \varphi$ and $\rho+1 \leqslant \varphi$.

Otherwise let $X$, $Y$, $Z$, and $W$ be orders of types $\rho$, $\tau$, $\varphi$, and $\psi$, respectively. Let us assume that $\tau \isntcontainedin \psi$. Let $e : XY \longrightarrow ZW$ be an embedding and, for $y\in Y$, denote by $X_y$ the $y$-th copy of $X$ in the product $XY$ and, similarly, for $w \in W$ by $Z_w$ the $w$-th copy of $Z$ in $ZW$.
For every $y \in Y$ let $E(y) = \{w \in W : e(X_y) \cap Z_w \neq \emptyset\}$, which is nonempty because $X_y \neq \emptyset$.

Now define $f: Y \to W$ so that, for every $y \in Y$ we have $f(y) \in E(y)$, with the proviso that, if $|E(y)| > 1$ and $E(y)$ has a minimum, then $f(y)$ is not that minimum. 
Since $y < y'$ implies $f(y) \leq f(y')$ and  $\tau \isntcontainedin \psi$ we must have $f(y) = f(y')$ for some $y<y'$.
This requires $f(y')$ to be the minimum of $E(y')$ and, by our choice of $f(y')$, this implies $|E(y')|=1$, so that $e$ maps $X_{y'}$ into $Z_{f(y')} = Z_{f(y)}$.
Pick $x \in X_y$ such that $e(x) \in Z_{f(y)}$.
Then the restriction of $e$ to $\{x\} \cup X_{y'}$ witnesses $1+X_{y'} \leqslant Z_{f(y)}$ and hence $1+\rho \leqslant \varphi$.

The proof that $\rho+1 \leqslant \varphi$ is similar, avoiding the maximum of $E(y)$ whenever possible. 
\end{proof}

We will use \Cref{new-product-lemma} only as the following the multiplicative analog\ue\ of \Cref{sum-inequality-separation}.

\begin{coro}
\label{product-lemma}
Suppose $\rho, \tau, \varphi$ and $\psi$ are order types. If $\rho \tau \leqslant \varphi \psi$ then either $\rho \leqslant \varphi$ or $\tau \leqslant \psi$.
\end{coro}
\if
\todo[inline, color = thilo]{
I wonder whether the following strengthening is true:

Suppose $\rho, \tau, \varphi$ and $\psi$ are order types. If $\rho \tau \leqslant \varphi \psi$ then either $\rho = 0 = \varphi$ or $1 + \rho \leqslant \varphi$ or $\tau \leqslant \psi$.

\todo[inline]{I think this is true if you change ``or $1 + \rho \leqslant \varphi$ or" to ``or $1 + \rho \leqslant \varphi$ or $\rho + 1 \leqslant \varphi$ or", though it might be true even for ``or $(1 + \rho \leqslant \varphi$ and $\rho + 1 \leqslant \varphi)$ or"; not sure about ``or $1 + \rho + 1 \leqslant \varphi$ or"; but ``or $\rho 2 \leqslant \varphi$ or" is false, which is in some sense the most natural analogue. }

\begin{proof}
Let $X$, $Y$, $Z$, and $W$ be orders of types $\rho$, $\tau$, $\varphi$, and $\psi$, respectively. Let us assume that $\rho \isntcontainedin \varphi$. Let $e : XY \longrightarrow ZW$ be an embedding. As $\rho \isntcontainedin \varphi$, we know that for every $y \in Y$ there exist $x, x' \in X$, $p_y, q_y \in W$ with $p_y < q_y$, and $z, z' \in Z$ such that $e(x, y) = \opair{z}{p_y}$ and $e(x', y) = \opair{z'}{q_y}$. We define a map
\begin{align*}
f : Y & \longrightarrow W, \\
y & \longmapsto q_y.
\end{align*}
As for $y, y' \in Y$, the condition $y < y'$ implies $q_y \leq p_{y'} < q_{y'}$, $f$ is an embedding.
\end{proof}
\fi

\begin{coro}
\label{embedding one's square}
Suppose that $\varphi$ is an order type such that $\varphi^2 \equim \varphi$. Then $\varphi$ is both $s$-untranscendable and product closed. In particular, $\varphi$ is untranscendable.
%\todo[inline, color = thilo]{It is also strongly indecomposable.}
\end{coro}

\begin{proof}
Suppose $\psi$ and $\tau$ are types such that $\varphi \leqslant \psi \tau$. Then $\varphi \varphi \leqslant \psi \tau$ so that by \Cref{product-lemma} either $\varphi \leqslant \psi$ or $\varphi \leqslant \tau$. Hence $\varphi$ is $s$-untranscendable. 

For product closure, suppose $\psi$ and $\tau$ are types such that $\psi < \varphi$ and $\tau < \varphi$. Then $\psi \tau \leqslant \varphi \varphi \leqslant \varphi$. If it were the case that $\varphi \leqslant \psi \tau$, then we would have $\varphi \varphi \leqslant \psi \tau$, giving either $\varphi \leqslant \psi$ or $\varphi \leqslant \tau$, a contradiction. Hence $\psi \tau < \varphi$. 
\end{proof}

\subsection{Examples of untranscendable types} As we observed in \Cref{IndecTypes}, the empty type $0$ and singleton type $1$ are both indecomposable. They are also untranscendable, as follows from the definition. The two-point type $2$ is not indecomposable. However, it \emph{is} untranscendable since if $\psi, \tau < 2$ then $\psi, \tau \leqslant 1$ and hence $\psi\tau \leqslant 1 < 2$. That is, $2$ is product closed. It is also clear that $2$ is $s$-untranscendable. In \Cref{AD+UT=2} we will show that $2$ is quite special among the untranscendable types: it is the unique untranscendable type that is not indecomposable. 

No larger finite type $n$ is untranscendable. On the other hand, $\omega$ is untranscendable, and in fact both product closed and $s$-untranscendable. More generally, we have the following characterization of untranscendable ordinals:

\begin{prop}
\label{untrans_ord}
Suppose $\alpha$ is an ordinal. Each of the following statements implies the next one and 
%If $\alpha \ne 1$, then \ref{exponent} is equivalent to \ref{delta}. If $\alpha > 0$, then \ref{delta} is equivalent to \ref{idc&pc}. If $\alpha \ne 2$, then \ref{idc&pc} is equivalent to \ref{pc}.
if $\alpha > 2$, then they are all equivalent\footnote{For $\alpha \leq 2$, note that $1$ satisfies \eqref{delta}, \eqref{idc&pc}, \eqref{pc}, and \eqref{ut}; $0$ satisfies \eqref{idc&pc}, \eqref{pc}, and \eqref{ut}; and $2$ satisfies \eqref{pc} and \eqref{ut}.}
\begin{enumerate}[label=(\roman*)]
    \item \label{exponent}$\alpha = \omega^{\omega^{\beta}}$ for some ordinal $\beta$,
    \item \label{delta}$\alpha$ is a $\delta$-number,
    \item \label{idc&pc}$\alpha$ is both indecomposable and product closed,
    \item \label{pc}$\alpha$ is product closed,
    \item \label{ut}$\alpha$ is untranscendable.
\end{enumerate}
\end{prop}

\begin{proof}
Recall that $\delta$-numbers are indecomposable of the form $\omega$ raised to an indecomposable power. Therefore the implication $\ref{exponent} \Rightarrow \ref{delta}$  holds due to powers of $\omega$ being indecomposable, cf.\ \Cref{indecexs}.(iii).

The first half of \ref{idc&pc} follows due to the same fact. Suppose that $\omega^\vartheta$ is a $\delta$-number. If $\vartheta \leqslant 1$ then the product-closure is immediate. Otherwise $\vartheta$ is a limit ordinal and if $\gamma, \delta < \omega^\vartheta$ and $\gamma_0$ and $\delta_0$ are the largest exponents in their Cantor normal forms, we have $\gamma_0 + 1, \delta_0 + 1 < \vartheta$. Therefore
\[
\gamma \delta \leq \omega^{\gamma_0+1} \omega^{\delta_0+1} = \omega^{\gamma_0+1+\delta_0+1} < \omega^\vartheta,
\]
where in the last step we use indecomposability of $\vartheta$ (\Cref{indecexs}.(iii)).

\ref{idc&pc} implies \ref{pc} and \ref{pc} implies \ref{ut} are trivial.

Finally, to show that \ref{ut} implies \ref{exponent}, notice that $\alpha$ can be written as
\[
\omega^{\omega^\beta k + \gamma}n + \delta
\]
where $\delta < \omega^{\omega^\beta k + \gamma}$, $\gamma < \omega^\beta$, and $0< k, n < \omega$.

If $\max(n,\gamma+1,\delta+1) \geqslant 2$ then $\omega^{\omega^\beta k} < \alpha$ and $\alpha \leqslant \omega^{\omega^\beta2k} = \left(\omega^{\omega^\beta k}\right)^2$, witnesses that $\alpha$ is transcendable. 
If $\max(n,\gamma+1,\delta+1) = 1$ and $k \geq 2$, $\omega^{\omega^\beta(k - 1)} < \alpha$ and $\alpha \leq \omega^{\omega^\beta2(k - 1)} = \left(\omega^{\omega^\beta(k - 1)}\right)^2$, witnesses $\alpha$'s transcendability. 
Thus $n=k=1$ and $\gamma = \delta = 0$, so that $\alpha = \omega^{\omega^\beta}$.
\end{proof}

Thus the least untranscendable ordinal greater than $\omega$ is $\omega^{\omega} = \omega + \omega^2 + \omega^3 + \cdots = \sum_{n \in \omega} \omega^n$. \Cref{untranscendability of sums of finite powers} below shows that this is an instance of a more general construction for building untranscendable types: indeed, summing the finite powers of \textit{any} fixed order type $\varphi$ always yields an untranscendable type. (See \Cref{ut_but_not_sut_ord<=>sing_lim_ord} for a characterization of the untranscendable ordinals.) First we need a couple of lemmas.

\begin{lemm}
\label{sum-of-powers-lemma}
Let $\rho$ and $\varphi$ be any order types. If $1 + \rho\varphi \leqslant \varphi$, then
\begin{align*}
\tau = \sum_{n < \omega} \rho^n \leqslant \varphi.
\end{align*}
\end{lemm}

\begin{proof}
Let $X$ and $Y$ be orders of type $\rho$ and $\varphi$ respectively, and let $e : \{a\} + XY \longrightarrow Y$ be an embedding witnessing that $1 + \rho\varphi \leqslant \varphi$. We recursively define a map
\begin{align*}
f : \sum_{n < \omega} X^n & \longrightarrow Y, \\
x & \longmapsto \begin{cases}
e(a) & \text{ if } x \in X^0, \\
e(\opair{x_0}{f(x_1, \dots, x_n)}) & \text{ if } x = \langle x_0, x_1, \dots, x_n \rangle \in X^{n + 1} \text{ else.}
\end{cases}
\end{align*}
One can prove via an induction argument that $e$ is an embedding.
This embedding witnesses that $\tau \leqslant \varphi$ via the nested sum $\tau = 1 + \rho(1 + \rho(1 + \cdots))$.
\end{proof}

\begin{lemm}\label{Alberto's-lemma}
Suppose $m \geq 1$, and $\varphi_0, \varphi_1, \ldots, \varphi_m$ and $\psi_0, \psi_1, \ldots, \psi_{m-1}$ are order types with $\varphi_i \neq 0$ for all $i < m + 1$. 

If $\sum_{i<m+1} \varphi_i \leqslant \sum_{i<m} \psi_i$, then $1+\varphi_{i+1} \leqslant \psi_i$ for some $i<m$.
\end{lemm}

\begin{proof}
By induction on $m$.
\end{proof}

\begin{prop}
\label{untranscendability of sums of finite powers}
Suppose $\rho$ is any order type. Then both the type
\[
\tau = \sum_{n \in \omega} \rho^n = 1 + \rho + \rho^2 + \cdots
\]
and the type
\[
\tau' = \sum_{n \in \omega^*} \rho^n = \cdots + \rho^2 + \rho + 1 
\]
are untranscendable. 
\end{prop}

\begin{proof}
For reasons of symmetry it is sufficient to prove only the first assertion. If $\rho$ is finite, then either $\tau = 1$ (when $\rho = 0$) or $\tau = \omega$, and in both cases $\tau$ is untranscendable. 

From now on we assume that $\rho$ is infinite, and that $\varphi \leqslant \tau$ and $\psi \leqslant \tau$ are such that $\tau \leqslant \varphi\psi$. 

If $1 + \rho^{1 + n} \leqslant \rho^n$ for some $n$, by \Cref{sum-of-powers-lemma} we have $\tau \leqslant \rho^n$.
Since $\rho^{2n} \leqslant \tau \leqslant \varphi\psi$ by \Cref{product-lemma} we have either $\rho^n \leqslant \varphi$ or $\rho^n \leqslant \psi$, yielding either $\tau \leqslant \varphi$ or $\tau \leqslant \psi$, as needed.

We now assume $1 + \rho^{1 + n} \isntcontainedin \rho^n$ for every $n$.
Since $\rho$ is infinite, then for every $n$, $\rho^n (n+1) \rho^n (n+1) \leqslant \rho^{4n} \leqslant \tau \leqslant \varphi\psi$. 
By \Cref{product-lemma} we must have either $\rho^n (n+1) \leqslant \varphi$ or $\rho^n (n+1) \leqslant \psi$.
As $\rho^n (n+1) \leqslant \varphi$ implies $\rho^k (k+1) \leqslant \varphi$ for every $k \leqslant n$ (and the analogous statement with $\psi$ in place of $\varphi$ holds true as well), we have that either $\rho^n (n+1) \leqslant \varphi$ for all $n$ or $\rho^n (n+1) \leqslant \psi$ for all $n$. 
Let us assume that the latter is the case (the  argument for the other case is analogous).

Then for any order $X$ of type $\psi$ and any $n$, any order of type $\rho^n (n+1)$ can be embedded into $X$ and in fact into an initial segment of it (using the embedding of $\rho^{n+1} (n+2)$ into $X$). 
Let $Y$ be an order of type $\rho$ and fix $X \subset \sum_{n < \omega} Y^n$ of type $\psi$ ($X$ exists because $\psi \leqslant \tau$). 
We inductively define an embedding
\begin{align*}
e : \sum_{n < \omega} Y^n & \longrightarrow  X
\intertext{thus showing $\tau$ to be untranscendable.
Suppose that at stage $k$ of the construction we defined an embedding}
e_k : \sum_{n < k} Y^n & \longrightarrow X \subset \sum_{n < \omega} Y^n
\end{align*}
into an initial segment of $X$. Let $m \geq k$ be such that the image of $e_k$ is included in $\sum_{n < m} Y^n$.
Since $\rho^{m+1} (m+2) \leqslant \psi$ there exists an embedding $f: Y^{m+1} (m+2) \longrightarrow X$ witnessing this. 
By \Cref{sum-inequality-separation}, since $X = (X \cap \sum_{n < m} Y^n) + (X \cap \sum_{n \geq m} Y^n)$, $f$ witnesses either $Y^{m+1} (m+1) \leqslant X \cap \sum_{n < m} Y^n$ or $Y^{m+1} \leqslant X \cap \sum_{n \geq m} Y^n$.
By \Cref{Alberto's-lemma} the first case implies $1 + Y^{m+1} \leqslant Y^n$ for some $n<m$, contradicting $1 + \rho^{m+1} \isntcontainedin \rho^m$.
Therefore $f$ maps a set of order type $\rho^{m+1}$ into $X \cap \sum_{n \geq m} Y^n$ and a fortiori a set of order type $\rho^k$ into an initial segment of
$X \cap \sum_{n \geq m} Y^n$. 
Now we can use $f$ to extend $e_k$ to an embedding $e_{k + 1}$ of $\sum_{n < k + 1} Y^n$ into an initial segment of $X$.
\end{proof}

\begin{examples}\label{ex:untr}
\phantom{}
\begin{enumerate}[leftmargin=*,label=(\roman*)]
    \item\label{ex:untr.not s-untr} By \Cref{untranscendability of sums of finite powers}, the sums
    \[
    \Phi = 1 + \omega^* + (\omega^*)^2 + \cdots = \sum_{n \in \omega} (\omega^*)^n
    \]
    and 
    \[
    \Phi^* = \cdots + \omega^2 + \omega + 1 = \sum_{n \in \omega^*} \omega^n
    \]
    are untranscendable. These are scattered types that are neither ordinals nor reverse ordinals. 
    
    We claim that $\Phi$ and $\Phi^*$ are neither $s$-untranscendable nor product closed. Let us check this for $\Phi^*$. Since $\omega^n$ embeds in $\omega^{\omega}$ for every $n$, we have that $\Phi^* \leqslant \cdots + \omega^{\omega} + \omega^{\omega} = \omega^{\omega} \omega^*$. Clearly, $\Phi^* \isntcontainedin \omega^{\omega}$ and $\Phi^* \isntcontainedin \omega^*$, since $\Phi^*$ is neither well-ordered nor reverse well-ordered. Thus $\Phi^*$ is not $s$-untranscendable. To see that $\Phi^*$ is not product closed, observe that since $\Phi^*$ is an $\omega^*$-sum of ordinals, any infinite descending sequence in $\Phi^*$ (that is, any suborder of $\Phi$ isomorphic to $\omega^*$) must be unbounded to the left in $\Phi^*$. It follows that $\omega^* 2 \isntcontainedin \Phi^*$. Since both $2 < \Phi^*$ and $\omega^* < \Phi^*$, this shows $\Phi^*$ is not product closed.

    Interestingly, these order types are also instrumental in showing that not even all initial ordinals are $s$-untranscendable, cf.\ \Cref{ut_but_not_sut_ord<=>sing_lim_ord} below.
    
    \item \Cref{embedding one's square} gives that if $\varphi$ is a type embedding its square $\varphi^2$, then $\varphi$ is not only untranscendable but in fact $s$-untranscendable and product closed.
    Thus $\eta$ is $s$-untranscendable and product closed. 
    \item Since $\mathbb{R}$ is separable, any collection of pairwise disjoint intervals in $\mathbb{R}$ must be countable. It follows that $2 \lambda \isntcontainedin \lambda$ and, a fortiori, $\lambda^2 \isntcontainedin \lambda$. Thus we cannot use \Cref{embedding one's square} to prove that $\lambda$ is untranscendable. However, $\lambda$ \emph{is} untranscendable. In fact, $\lambda$ is $s$-untranscendable. This follows from a more general fact; see \Cref{homogeneity =>s-untranscendability}. 
\end{enumerate}    
\end{examples}

\begin{defi}
A linear order $X$ is \emph{homogeneous} if whenever $I \subseteq X$ is a non-degenerate interval in $X$, then $X \leqslant I$.

An order type $\varphi$ is \emph{homogeneous} if some (equivalently every) linear order $X$ of type $\varphi$ is homogeneous. 
\end{defi}

The types $0$, $1$, and $2$ are homogeneous. Any other homogeneous type must be dense. For example, both $\eta$ and $\lambda$ are homogeneous types, since $\eta$ and $\lambda$ are isomorphic to all of their open intervals. 

\begin{prop}
\label{homogeneity =>s-untranscendability}
If $\varphi$ is a homogeneous type, then $\varphi$ is $s$-untranscendable. 
\end{prop}

\begin{proof}
Fix an order $X$ of type $\varphi$ and suppose we have orders $Y$ and $Z$ and an embedding $f: X \rightarrow ZY$. Given $y \in Y$, we write $Z_y$ for the $y$th copy of $Z$ in $ZY$, that is $Z_y = \{\opair{z}{y} \in ZY: z \in Z \}$. In particular, $Z_y \cong Z$. 

If for some $y \in Y$ there exist $x < x'$ in $X$ such that $f(x), f(x') \in Z_y$ then $Z_y$ contains the image of the non-degenerate interval $[x, x']$. 
By homogeneity, $X$ embeds in $[x, x']$ and hence in $Z_y$, so that $X \leqslant Z$. 
If instead for every $y \in Y$ there is at most one $x \in X$ such that $f(x) \in Z_y$ it is easy to define an embedding of $X$ in $Y$.
\end{proof}

It follows immediately from \Cref{homogeneity =>s-untranscendability} that $\lambda$ is $s$-untranscendable, despite the fact that $\lambda^2 \isntcontainedin \lambda$. However, $\lambda$ is not product closed. Sierpi\'nski showed that there exist suborders $A \subseteq \mathbb{R}$ of the same cardinality as $\mathbb{R}$ that do not embed $\mathbb{R}$ (cf.\ \cite[Th\'eor\`eme 1]{950S1} and also \cite[Theorem 9.10]{982R0}). Any such $A$ is in particular uncountable, so that $2A$ does not embed in $\mathbb{R}$ (by separability). Letting $\psi = \textrm{otp}\langle A \rangle$, we have $2 < \lambda$, $\psi < \lambda$, but $2 \psi \isntcontainedin \lambda$, giving that $\lambda$ is not product closed. 

While every homogeneous type is $s$-untranscendable by \Cref{homogeneity =>s-untranscendability}, not every $s$-untranscendable type is homogeneous. For example, $\omega$ and $\omega^*$ are not homogeneous, but they are $s$-untranscendable. 
%To see this for $\omega^*$, suppose that $\omega^* \leqslant \psi \tau$ for some order types $\psi$ and $\tau$. If it were the case that $\omega^* \isntcontainedin \psi$ and $\omega^* \isntcontainedin \tau$, then both $\psi$ and $\tau$ are well-ordered types, i.e., ordinals. Since the product of two ordinals is also an ordinal, it follows $\omega^* \isntcontainedin \psi \tau$, a contradiction. 

\subsection{A Hagendorf-Jullien type theorem for $s$-untranscendable types} Our goal in this section is to prove an analog\ue\ for $s$-untranscendable types of the Hagendorf-Jullien theorem (\Cref{Hagendorf's} above) for indecomposable types. 

By \Cref{containing two copies}, if $\varphi$ is a type embedding $\varphi + \varphi$, then $\varphi$ is indecomposable, and $\varphi$ embeds both $\varphi + 1$ and $1 + \varphi$ (assuming $\varphi$ is non-empty). The Hagendorf-Jullien theorem says precisely that the converse is true: if $\varphi$ is indecomposable and embeds both $1 + \varphi$ and $\varphi + 1$, then $\varphi$ embeds $\varphi + \varphi$ (cf.\ \Cref{1+phi<=phi + phi+1<=phi  ==>  phi+phi<=phi}). 

Analogously, by \Cref{embedding one's square}, if $\varphi$ is a type embedding $\varphi\varphi$, then $\varphi$ is untranscendable (in fact, $s$-untranscendable and product closed), and $\varphi$ embeds both $\varphi 2$ and $2 \varphi$ (assuming $\varphi$ is not the singleton type $1$ or the zero type $0$). Our analog\ue\ of Hagendorf's result is the following theorem, which says that the converse is true. 

\begin{theo}
\label{Garrett's Theorem}
If $\varphi$ is $s$-untranscendable and both $\varphi 2 \leqslant \varphi$ and $2 \varphi \leqslant \varphi$, then $\varphi^2 \leqslant \varphi$. 
\end{theo}

Observe that no two of the hypotheses in \Cref{Garrett's Theorem} are sufficient to derive the conclusion. Indeed, $\varphi = \otR$ satisfies $s$-untranscendability and $\varphi 2 \leqslant \varphi$; $\varphi = \omega$ satisfies $s$-untranscendability and $2 \varphi \leqslant \varphi$; and $\varphi = \omega \lambda$ satisfies $\varphi 2 \leqslant \varphi$ and $2 \varphi \leqslant \varphi$; but none of these orders satisfy $\varphi^2 \leqslant \varphi$. 

We do not know if \Cref{Garrett's Theorem} remains true if we assume untranscendability in place of $s$-untranscendability. See \Cref{2x=x=x2 =?=> x=xx} below. 

We will prove \Cref{Garrett's Theorem} in a sequence of lemmas.

\begin{lemm}
\label{homogeneity + 2x<=x  ==>  xx<=x}
If $\varphi$ is a homogeneous type and $2 \varphi \leqslant \varphi$, then $\varphi^2 \leqslant \varphi$.
\end{lemm}

\begin{proof}
The hypothesis $2 \varphi \leqslant \varphi$ implies that we can find a collection of $\varphi$-many pairwise disjoint non-degenerate intervals in $\varphi$. By homogeneity, $\varphi$ embeds in each of these intervals, and it follows $\varphi^2 \leqslant \varphi$. 
\end{proof}

To prove \Cref{Garrett's Theorem} we will show that if $\varphi$ is $s$-untranscendable and $\varphi 2 \leqslant \varphi$, then $\varphi$ is equimorphic to a homogeneous type. It then follows from \Cref{homogeneity + 2x<=x  ==>  xx<=x} that $\varphi^2 \leqslant \varphi$.

\Cref{selfsimcondensationlemma} below indicates how the hypothesis $\varphi 2 \leqslant \varphi$ will be used. The idea is that if $\varphi 2 \leqslant \varphi$, then we can get a condensation of $\varphi$ by condensing the maximal intervals in $\varphi$ that do not embed a copy of $\varphi$, and the resulting condensed order is homogeneous. 
It is convenient to state the lemma in terms of orders rather than order types. 

\begin{lemm} \label{selfsimcondensationlemma}
Suppose that $X$ is a linear order such that $X2 \leqslant X$. Define a binary relation $E$ on $X$ by the rule $x E x'$ if and only if $X \isntcontainedin [\{x, x'\}]$.

Then $E$ is a condensation on $X$, no $E$-equivalence class embeds $X$, and the induced order $X / E$ is homogeneous. 
\end{lemm}

\begin{proof}
We first show that $E$ is a condensation, i.e.\ an equivalence relation on $X$ whose equivalence classes are intervals. 

It is clear that if $x \leq y \leq z$ and $x E z$ then $x E y$.
It remains to show $E$ is an equivalence relation. It follows immediately from the definition that $E$ is reflexive and symmetric. 

To see that $E$ is transitive, suppose that $x, y, z \in X$ and $x E y E z$. We wish to show $x E z$, that is, that the closed interval between $x$ and $z$ does not embed $X$. There are six possible ways the points $x, y, z$ might be ordered with respect to one another. By the convexity of $E$, it suffices to consider the two cases when $y$ is between $x$ and $z$. Then without loss of generality we may assume $x \leq y \leq z$. If $X$ embeds in $[x, z]$, then since $X + X$ embeds in $X$ we have that $X + X$ embeds in $[x, z] = [x,y] + (y,z]$. 
By \Cref{sum-inequality-separation} $X$ embeds in either $[x,y]$ or $(y,z]$, contradicting $x E y$ or $y E z$.

To see that each condensation class $I \in X / E$ does not embed $X$, observe that since $X + X$ embeds in $X$, we have that $X + X + X$ embeds in $X$ as well. If such an $I$ did embed $X$, then $I$ embeds $X + X + X$. Viewing $X + X + X$ as a suborder of $I$, we can find an $x \in I$ from the initial copy of $X$ and a $y \in I$ from the final copy of $X$. But then $[x, y]$ embeds the middle copy of $X$, contradicting $x E y$. 

On the other hand, if $x, y \in X$ and $x\un{E}y$, then $[x, y]$ embeds $X$ by definition of $E$. This shows that the condensation classes $I \in X / E$ are maximal (as intervals in $X$) with respect to not embedding $X$. 

Now we prove the homogeneity of $X / E$. As discussed in \Cref{HausdorffCondSection}, we can view $X$ as an ordered sum over $X / E$ by the condensation classes of $E$. We write $X = \sum_{I \in X / E} I$, where in the subscript we are viewing each $I$ as a point in $X / E$, and in the sum as an interval in $X$. 

Fix points $A < B$ in $X / E$, where $A = [x]$ and $B = [y]$ are represented by points $x < y$ in $X$. We wish to show that $X / E$ embeds in $[A, B]$. 

Since $x \un{E} y$ we know that $X$ embeds in $[x, y]$. Fix an embedding $f: X \rightarrow [x, y]$. Observe that if $q < r$ are points in $X$ such that $q\un{E}r$, then we must have $f(q)\un{E}f(r)$, as otherwise we would get an embedding of $X$ into an $E$-class. Thus if for every condensation class $I \in X / E$ we pick a point $q_I$ in the image $f[I]$, then we have $q_I\un{E}q_J$ whenever $I \neq J$. Then the map $I \mapsto [q_I]$ is an embedding of $X / E$ into $[A, B]$, as desired. 
\end{proof}

The hypothesis $X2 \leqslant X$ implies that $X$ is infinite and hence $X / E$ from \Cref{selfsimcondensationlemma} is also infinite. It  follows that since $X / E$ is homogeneous, it is dense.

The next lemma will essentially finish the proof of \Cref{Garrett's Theorem}. It says that for order types $\varphi$ embedding $\varphi2$, $s$-untranscendability is equivalent to homogeneity up to equimorphism.
More explicitly, if $\varphi$ is a homogeneous type then by \Cref{homogeneity =>s-untranscendability}, $\varphi$ is $s$-untranscendable. Moreover, assuming $\varphi$ is not one of the three finite homogeneous types, then $\varphi 2 \leqslant \varphi$, since by homogeneity any pair of non-overlapping intervals in $\varphi$ yields an embedding of $\varphi 2$ into $\varphi$. \Cref{s-untranscendable + phi2 <= phi => homogeneous} below says the converse is true, up to equimorphism: if $\varphi$ is $s$-untranscendable and $\varphi 2 \leqslant \varphi$, then $\varphi$ is equimorphic to a homogeneous type.

\begin{lemm}\label{s-untranscendable + phi2 <= phi => homogeneous}
Suppose $\varphi$ is an $s$-untranscendable type such that $\varphi 2 \leqslant \varphi$. Then there is a homogeneous type $\varphi'$ such that $\varphi' \equim \varphi$. 
\end{lemm}

\begin{proof}
Fix a linear order $X$ of type $\varphi$. Then $X$ embeds $X2 = X + X$. Let $E$ be the condensation of $X$ from \Cref{selfsimcondensationlemma}. 
We will prove $X/E$ is equimorphic with $X$. Clearly $X/E$ embeds in $X$, so it suffices to show that $X$ embeds in $X/E$. 

Suppose $X / E$ has cardinality $\kappa$. Enumerate $X / E$ as $\{I_{\alpha}: \alpha \in \kappa\}$. Viewing $\kappa$ as an ordinal, consider the sum $Y = \sum_{\alpha \in \kappa} I_{\alpha}$. Since $X = \sum_{I \in X / E} I$ and every $I \in X / E$ appears as an interval in $Y$, we clearly have that $X$ embeds in $Y (X / E)$. 
By the $s$-untranscendability of $X$ (i.e., of $\varphi = \textrm{otp} \langle X \rangle$), we have that $X$ embeds in either $Y$ or $X/E$. 
We will show that $X$ does not embed in $Y$, which will finish the proof of the lemma.

Suppose there were an embedding $f: X \rightarrow Y$, i.e.\ an embedding 
\[
f: \sum_{I \in X / E} I \rightarrow \sum_{\alpha \in \kappa} I_{\alpha}.
\]
There are two cases. If $X / E$ does not have a left endpoint (i.e.\ $X$ does not have a leftmost $E$-class), then let $\alpha$ be least in $\kappa$ such that the image $f[X]$ intersects $I_{\alpha}$. If $X / E$ has a left endpoint $I$ (i.e.\ $I$ is the leftmost $E$-class in $X$), let $\alpha$ instead be least in $\kappa$ such that $f[X \setminus I]$ intersects $I_{\alpha}$. In either case, since $X/E$ or $(X \setminus I) /E$ has no left endpoint (recall that $X/E$ is dense) it must be that $I_{\alpha}$ intersects the image of two distinct $E$-classes, i.e.\ there are $I < J$ in $X / E$ such that $f[I]$ and $f[J]$ both intersect $I_{\alpha}$. We can choose points $x \in I$ and $y \in J$ such that the image of $[x, y]$ lies in $I_{\alpha}$. Since $x\un{E}y$, $X$ embeds into $[x, y]$ and hence into $I_{\alpha}$,  contradicting \Cref{selfsimcondensationlemma}. 
\end{proof}

\emph{Proof of \Cref{Garrett's Theorem}}:
Fix an order $X$ of type $\varphi$. By \Cref{s-untranscendable + phi2 <= phi => homogeneous}, $X$ is equimorphic to a homogeneous type $X/E$. Let $\varphi'$ be the type of $X/E$. Then since $2 \varphi \leqslant \varphi$ we also have $2 \varphi' \leqslant \varphi'$. By \Cref{homogeneity + 2x<=x  ==>  xx<=x}, since $\varphi'$ is homogeneous, we have $(\varphi')^2 \leqslant \varphi'$. But then $\varphi^2 \leqslant \varphi$ as well. \qed

\section{Untranscendability and indecomposability} \label{Untr+Indec}
In this section we investigate the relationship between untranscendability and indecomposability. We will show that, with the unique exception of the two-point type, every untranscendable type is indecomposable. We also show that untranscendable types that satisfy a certain extra condition, that of being equimorphic to a type containing only finitely many finite $F$-classes, are in fact strongly indecomposable (again with the exception of $2$). We then use this result to show that for several canonical hierarchies of order types, including the $\sigma$-scattered order types as well as the Aronszajn types under \PFA, untranscendability implies strong indecomposability. 

\subsection{Untranscendability implies indecomposability, except for 2}

\begin{lemm}
\label{lemma : psi(1+tau)}
Let $\psi$ and $\tau$ be order types. If $\psi\ne 0$, we have $\psi + \tau \leqslant \psi(1 + \tau)$. Similarly, if $\tau \neq 0$, then $\psi + \tau \leqslant \tau(\psi + 1)$.
\end{lemm}

\begin{proof}
Since the multiplication of linear order types is right-distributive, we have $\psi(1 + \tau) = \psi + \psi\tau$. Therefore, to prove the first statement, it suffices to show that $\tau \leqslant \psi\tau$. This is true, because $1 \leqslant \psi$. The second statement can be proved analogously.
\end{proof}

Using this lemma, we can prove the following:

\begin{prop}
\label{AD+UT=2}
The finite linear order type of cardinality $2$ is the only linear order type which is decomposable yet untranscendable.
\end{prop}

\begin{proof}
Assume that $\psi + \tau$ is untranscendable but decomposable, witnessed by $\psi$ and $\tau$, so in particular $\psi + \tau \isntcontainedin \psi $ and $\psi + \tau \isntcontainedin \tau$. The latter implies that $1 \leqslant \psi$. As $\psi + \tau$ is untranscendable, Lemma \ref{lemma : psi(1+tau)} implies that $\psi + \tau \leqslant 1 + \tau$. Were $2 \leqslant \psi$ true, then we would have $2 + \tau \leqslant 1 + \tau$. As $2 \isntcontainedin 1$ we have $1 + \tau \leqslant \tau$ by \Cref{strong-sum-inequality-separation}.  We obtain $\psi + \tau \leqslant \tau$ by transitivity, a contradiction. Therefore, $\psi = 1$. Similarly one can show that $\tau = 1$. So $\psi + \tau = 2$.
\end{proof}

\begin{coro}
\label{corollary : No Sum}
Suppose that $\varphi$ is an order type that can be written as a sum
\[
\varphi = \sum_{i = 0}^n \psi_i 
\]
where $n \in \mathbb{N}$, $n \geq 2$, and for all $i \leq n$, $\psi_i \neq 0$ and $\varphi \isntcontainedin \psi_i$. Then $\varphi$ is transcendable.
\end{coro}
\begin{proof}
Since $n \geq 2$ we have $3 \leqslant \varphi$ and thus $\varphi$ is not $2$. 
Therefore, by \Cref{AD+UT=2}, it suffices to show that $\varphi$ is decomposable.
This follows because a straightforward induction on $n$ shows that for every $n \geq 1$ if $\sum_{i = 0}^n \psi_i$ is indecomposable then $\varphi \leqslant \psi_i$ for some $i \leq n$.
\end{proof}

One may, reading \Cref{untrans_ord}, wonder whether untranscendability is equivalent to $s$-untranscendability for all ordinals or at least up to a set of exceptions. In fact we can characteri\z e the class of untranscendable ordinals which fail to be $s$-untranscendable but it turns out that it is a proper one:

\begin{prop}
\label{ut_but_not_sut_ord<=>sing_lim_ord}
    An ordinal is $s$-untranscendable if and only if it is untranscendable but no singular limit.
\end{prop}

\begin{proof}
We are first going to show that no singular limit ordinal is $s$-untranscendable. To that end, suppose that $\alpha$ is a singular limit ordinal of cofinality $\beta$ and let $\seq{\alpha_\nu}{\nu < \beta}$ be a cofinal sequence in $\alpha$. We define
\begin{align*}
   \rho = \sum_{\nu \in \beta^*} \alpha_\nu,  
\end{align*}
that is, we sum the $\alpha_\nu$'s in reverse order.
Clearly, $\alpha\leqslant\rho\beta$ and  $\alpha \isntcontainedin \beta$. Since the image of any embedding of an ordinal into $\rho$ intersects only finitely many $\alpha_\nu$'s we also have $\alpha \isntcontainedin \rho$. This shows that $\alpha$ is not $s$-untranscendable.

It remains to show that every untranscendable ordinal which is no singular limit is $s$-untranscendable. To that end, suppose that $\alpha$ is untranscendable but not $s$-untranscendable. Then there are types $\rho$ and $\varphi$ such that $\alpha \leqslant \rho\varphi$ but neither $\alpha \leqslant \rho$ nor $\alpha \leqslant \varphi$. This immediately implies $2 \leqslant \rho, \varphi$ and $3 \leqslant \alpha$.
Therefore, as $\gamma + 1\leqslant\gamma^2$ for all ordinals $\gamma > 1$ and $\alpha$ is untranscendable, $\alpha$ has to be a limit ordinal. Let $\seq{\alpha_\nu}{\nu < \cf(\alpha)}$ be a cofinal sequence in $\alpha$. Let $A$, $R$ and $P$ be linear orders of types $\alpha$, $\rho$, and $\varphi$, respectively and let $e : A \longrightarrow R\times P$ witness $\alpha \leqslant\rho\varphi$. If a final segment of $\pwi{e}{A}$ lies in a single copy of $R$, then $\alpha$ is decomposable, as any indecomposable ordinal is indecomposable to the right. But this would contradict \Cref{AD+UT=2}. Therefore no final segment of $\pwi{e}{A}$ lies in a single copy of $R$. Let $p$ be the projection of an an ordered pair to its right component and define $X_\xi = \pwi{(p \circ e)}{\alpha_{\xi+1}\setminus\alpha_\xi}$ for all $\xi < \cf(\alpha)$. Now we define inductively the following function $f$:
\begin{align*}
f : \cf(\alpha) & \longrightarrow P,\\
0 & \longmapsto (p\circ e)(0),\\
\nu & \longmapsto \min\left(\bigcup\soast{X_\xi}{\xi < \cf(\alpha) \wedge \forall\iota\left(\iota<\nu\rightarrow f(\iota) < \min(X_\xi)\right)}\right)
\end{align*}
It is easy to check that $f$ yields an embedding of $\cf(\alpha)$ into $P$ thus proving that $\cf(\alpha) \leqslant \varphi$. As $\alpha\isntcontainedin\varphi$ we know that $\alpha$ is singular. This finishes the proof.
\end{proof}

\subsection{Untranscendability and strong indecomposability}

In \Cref{2 only FF} below, we prove a sufficient condition for untranscendability to imply strong indecomposability. First, we extract from the proof of \Cref{AD+UT=2} a general method for finding suborders of an untranscendable order $X$ that embed $X$. 

\begin{defi}
Suppose that $X$ is a linear order and $Y \subseteq X$ is a non-empty suborder of $X$. Define a relation $E_Y$ on $X$ by the rule $x E_Y x'$ if and only if either $[\{x, x'\}] \subseteq Y$ or $x = x'$. 
\end{defi}

It is easy to see that the relation $E_Y$ is a condensation of $X$. Intuitively, $E_Y$ condenses each contiguous $Y$-segment in $X$ to a point. 

For a given $x \in X$, we denote the $E_Y$-class of $x$ by $[x]_Y$. We denote the condensed order $X / E_Y$ by $X /\!/ Y$. 

\begin{remark} \label{between E_Y eq classes}
By definition of $E_Y$, for any $C, C' \in X/\!/Y$ with $C < C'$ there is $x \in X \setminus Y$ such that $C \leq [x]_Y \leq C'$. Indeed, let $y \in C$ and $y' \in C'$: since $y$ and $y'$ are not $E_Y$-equivalent there exists $x \in [y,y'] \cap X \setminus Y$ and hence $C \leq [x]_Y \leq C'$.
\end{remark}

\begin{lemm} \label{Embed in Y or X//Y}
Suppose that $X$ is an untranscendable linear order and $Y \subseteq X$. Then either $X \leqslant Y$ or $X \leqslant X /\!/ Y$.  
\end{lemm}

\begin{proof}
If $Y = \emptyset$ then $X /\!/ Y \cong X$.
Otherwise, we claim $X \leqslant Y(X/\!/Y)$, i.e.\ $X \leqslant \sum_{[x]_Y \in X/\!/Y} Y$. 

To see this, fix $y_0 \in Y$. Define a map $f: X \rightarrow Y(X/\!/Y)$, as follows:
\begin{align*}
f(x) & = \begin{cases}
\opair{y_0}{[x]_Y} \text{ if } x \in X \setminus Y \\
\opair{x}{[x]_Y} \text{ if } x \in Y.
\end{cases}
\end{align*}
(Recall that we formally view $Y(X/\!/Y)$ as the set of ordered pairs $Y \times X /\!/ Y$, ordered anti-lexicographically.)

It is straightforward to check that $f$ is an embedding and hence
%To see this, suppose $x < x'$ are points in $X$. There are four cases to consider.
%
%If $x, x' \in X \setminus Y$, then $\{x\} = [x]_Y < [x']_Y = \{x'\}$ in $X /\!/ Y$. Hence $\opair{y_0}{[x]_Y} < \opair{y_0}{[x']_Y}$, i.e.\ $f(x) < f(x')$. 
%
%If $x \in X \setminus Y$ and $x' \in Y$, then $[x]_Y \neq [x']_Y$, and so $[x]_Y < [x']_Y$. Thus $\opair{y_0}{[x]_Y} < \opair{x'}{[x']_Y}$, i.e.\ $f(x) < f(x')$. The situation is symmetric if $x \in Y$ and $x' \in X \setminus Y$. 
%
%Finally, suppose $x, x' \in Y$. Then either $[x]_Y < [x']_Y$ or $[x]_Y = [x']_Y$. In the first case we have that $f(x) < f(x')$ is again witnessed by the second coordinate of these pairs. If $[x]_Y = [x']_Y$ then still we have $\opair{x}{[x]_Y} < \opair{x'}{[x']_Y} = \opair{x'}{[x]_Y}$, i.e.\ $f(x) < f(x')$, but now this is witnessed on the first coordinate. 
%
%Thus in all cases we have $f(x) < f(x')$, so that $f$ is an embedding, as claimed. 
$X \leqslant Y (X/\!/Y)$. 

On the other hand, we have $X/\!/Y \leqslant X$. Indeed, if for every $E_Y$-class $C$ we fix a representative $x_C \in C$ (so that $C = [x_C]_Y)$, then $[x_C]_Y \mapsto x_C$ defines an embedding of $X/\!/Y$ into $X$. 

Since clearly also $Y \leqslant X$, we have by the untranscendability of $X$ that either $X \leqslant Y$ or $X \leqslant X /\!/ Y$. 
\end{proof}

We will use \Cref{Embed in Y or X//Y} to prove the main result of this section, \Cref{2 only FF}. It says that any untranscendable type all of whose $F$-classes are infinite (cf.\ \Cref{HausdorffCondSection}) must be strongly indecomposable. More generally, any type equimorphic to such a type must be strongly indecomposable. 

The hypothesis of ``all $F$-classes infinite" is more natural than it perhaps appears.
Indecomposable types in certain classes (like the $\sigma$-scattered orders) are built up inductively using so-called regular unbounded sums and shuffles (we present an abstract version of such a construction in \Cref{abstract theorem}). These types are always equimorphic to a type none of whose $F$-classes are finite.

Some hypothesis beyond untranscendability is needed to guarantee strong indecomposability in general. We saw previously that $\lambda$ is untranscendable but not strongly indecomposable. See also the Trichotomy Conjecture \ref{realtypeconj}.

We will need the following lemma. 

\begin{lemm}
\label{AIs and finite CC}
Let $\varphi$ be an indecomposable order type, different from $1$. If $\varphi$ is equimorphic to a type with only finitely many finite $F$-classes, then $\varphi$ is equimorphic to a type with all $F$-classes infinite. 
\end{lemm}
\begin{proof}
Let $\varphi$ be indecomposable and suppose for some $m \in \mathbb{N}$ there is a type $\psi \equim \varphi$ with $m$ finite $F$-classes. As $\varphi$ is indecomposable, so is $\psi$. 

Suppose first that $\varphi \equim \varphi + \varphi$; then also $\psi \equim \psi + \psi$. Let $X$ be a linear order of type $\psi$. As $X2 \leqslant X$, we have by induction that $Xn \leqslant X$ for any $n \in \mathbb{N}$. In particular $X(2m+1) \leqslant X$. 

Let $f: X(2m+1) \rightarrow X$ be a fixed embedding. For each $i \leq 2m$, let $X_i$ denote the image under $f$ of the $i$th copy of $X$ in the sum $X(2m+1) = X + X + \cdots + X$. Let $X_i'$ denote the convex closure of $X_i$ in $X$. Clearly $X \equim X_i'$. 

Since each $X_i'$ is an interval in $X$, its $F$-classes are exactly the intersections of the $F$-classes of $X$ with $X_i'$. We claim that each $F$-class of $X$ intersects at most two of the intervals $X_i'$. Indeed, if $C$ were an $F$-class intersecting three of the intervals $X_i' < X_j' < X_k'$, then since $C$ is an interval we have $X_j' \subseteq C$. But then since $X_i' \cap C$ and $X_k' \cap C$ are non-empty, the interval $X_j'$ is bounded above and below in $C$, and hence must be finite, since $C$ is an $F$-class. But $X_j' \equiv X$ is infinite, a contradiction. Thus $C$ intersects at most two of the intervals $X_i'$, as claimed. 

Since there are only $m$-many finite $F$-classes in $X$, there must be an index $i_0 \leq 2m$ such that $X_{i_0}'$ has no finite $F$-classes. Let $\psi' = \textrm{otp} \langle X'_{i_0} \rangle$. Then $\psi' \equim \varphi$ and $\psi'$ has no finite $F$-classes, as desired. 

Now suppose $\varphi \not \equim \varphi + \varphi$. By the Hagendorf-Jullien Theorem, $\varphi$ is either strictly indecomposable to the right or to the left. Assume without loss of generality that $\varphi$, and hence also $\psi$, is strictly indecomposable to the right. 

Again fix an order $X$ of type $\psi$. Let $C$ denote the rightmost finite $F$-class in $X$, and let $x_{\textrm{max}}$ denote the right endpoint of $C$ (which exists, since $C$ is finite). It cannot be that $x_{\textrm{max}}$ is maximal in $X$, since $X$ is indecomposable to the right and not isomorphic to $1$. Let $Y = \{y \in X: y > x_{\textrm{max}}\}$. Then $Y$ is a non-empty final segment of $X$ without finite $F$-classes. By the right indecomposability of $X$, $Y \equim X$. Hence $\psi' = \textrm{otp}\langle Y \rangle$ has no finite $F$-classes and $\psi' \equim \varphi$. 
\end{proof}

\begin{theo}
\label{2 only FF}
Let $\varphi$ be an untranscendable order type that is equimorphic to one with only finitely many finite $F$-classes. Then $\varphi$ is strongly indecomposable or $\varphi = 2$.
\end{theo}

\begin{proof}
Suppose towards a contradiction that $\varphi \ne 2$ is untranscendable and equimorphic to a type with only finitely many finite $F$-classes, but $\varphi$ is not strongly indecomposable. By \Cref{AD+UT=2}, $\varphi$ must be indecomposable. As $1$ is strongly indecomposable, \Cref{AIs and finite CC} implies that $\varphi$ is equimorphic to a type $\psi$ without finite $F$-classes. By \Cref{strongindecequiinvar}, $\psi$ is not strongly indecomposable either. 

Let $\opair{X}{<}$ be a linear order of type $\psi$. Since $\psi$ is not strongly indecomposable, we can find $Y \subseteq X$ such that $X$ embeds in neither $Y$ nor $X \setminus Y$.

We consider the condensation $E_Y$: by \Cref{Embed in Y or X//Y} we have $X \leqslant Y$ or $X \leqslant X /\!/ Y$. Since $X \isntcontainedin Y$ by choice of $Y$, we must have $X \leqslant X /\!/ Y$. Let $f : X \longrightarrow X/\!/Y$ be an embedding. We will use $f$ to get an embedding $g$ of $X$ into $X \setminus Y$, which is a contradiction. 

We define $g$ separately on each $F$-class $D$ of $X$. Since $D$ is infinite, the order type of $D$ is one of $\omega$, $\omega^*$, or $\zeta$. Let $I$ denote either $\NN$, $ \{i \in \ZZ: i<0\}$ or $\ZZ$ so that we can enumerate $D$ in an order preserving way as $(x_i)_{i \in I}$.
For each $i \in I$, since $f(x_{2i}) < f(x_{2i+1})$, by \Cref{between E_Y eq classes} we can pick $x_i^* \in X \setminus Y$ such that $f(x_{2i}) \leq [x_i^*]_Y \leq f(x_{2i+1})$.
Let $g(x_i) = x_i^*$. Then $g$ is order-preserving: if $i,j \in I$ are such that $i<j$ we have $x_i^* \leq y_{2i+1} < y_{2j} \leq x_j^*$ for any  $y_{2i+1} \in f(x_{2i+1})$ and $y_{2j} \in f(x_{2j})$.
Moreover the convex closure of the image $g[D]$ is a subinterval of (the union over) the convex closure of $f[D]$. 
Therefore the definitions of $g$ on distinct $F$-classes do not clash. Thus $g$ is an embedding of $X$ into $X \setminus Y$, as promised.
\end{proof}

\subsection{Untranscendable types in hierarchies of regular unbounded sums and shuffles}

In this section we present two applications of \Cref{2 only FF}. These applications concern classes of types which are obtained by iterating certain kinds of sums and shuffles.

First, in his analysis of the $\sigma$-scattered types, cf.\ \cite[Theorem 3.1]{973L1} and \cite{971L0}, Laver showed that every indecomposable $\sigma$-scattered type can be built inductively by closing the family of types $\{0, 1\}$ under so-called \textit{regular unbounded sums} over regular cardinals $\kappa$ and their reverses, as well as under \textit{shuffles} over certain $\sigma$-scattered orders $\eta_{\alpha \beta}$ that generali\z e the rationals (see \Cref{regunboundedsumdef} and \Cref{shuffledef} below, as well as \Cref{Laver sigma-scattered indec thm}).

A second example uses axioms beyond \ZFC: under the the Proper Forcing Axiom ($\mathsf{PFA}$) one can define the class of \textit{fragmented Aronszajn types} (see \Cref{Aronszajnlinedef} and \Cref{fragmentedAlinedef}). Because this class is not closed under taking suborders, it is convenient to work with the class consisting all fragmented Aronszajn types along with all countable types. Following Barbosa, we denote this class by $\mathcal{C}$. Building on work of Moore \cite{006M1} and Martinez-Ranero \cite{011M2}, Barbosa showed \cite{023B0} that under the \PFA\, the indecomposable members of $\mathcal{C}$ are built inductively by closing $\{0, 1\}$ under countable regular unbounded sums and shuffles over $\QQ$ and the \textit{minimal Countryman lines} $C$ and $C^*$ (see \Cref{Countrymandef}).

It turns out that in both cases the pertinent combinatorial feature of the types $\varphi$ over which the shuffles are taken is that they are each equimorphic with their own square. In \Cref{abstractthm} below, we show that any class of types obtained by closing $\{0, 1\}$ under regular unbounded sums and shuffles over order types $\varphi$ such that $\varphi \equim \varphi^2$ only contains types that are equimorphic to types without finite $F$-classes, and $1$. We then use this result and \Cref{2 only FF}, along with the results of Laver and Barbosa, to show that the untranscendable $\sigma$-scattered linear orders, as well as the untranscendable Aronszajn lines under \PFA, are strongly indecomposable. 

\begin{defi} \label{regunboundedsumdef}
Suppose $\kappa$ is an infinite regular cardinal. 

A \emph{regular unbounded $\kappa$-sum} is an ordered sum of the form $\sum_{\alpha \in \kappa} \varphi_{\alpha}$, where for each $\alpha \in \kappa$ the set $\{\beta \in \kappa: \varphi_{\alpha} \leqslant \varphi_{\beta}\}$ is unbounded (to the right) in $\kappa$. 

A \emph{regular unbounded $\kappa^*$-sum} is an ordered sum $\sum_{\alpha \in \kappa^*} \varphi_{\alpha}$, where for each $\alpha \in \kappa^*$,  $\{\beta \in \kappa^*: \varphi_{\alpha} \leqslant \varphi_{\beta}\}$ is unbounded (to the left) in $\kappa^*$. 

A \emph{regular unbounded sum} is a regular unbounded $\kappa$-sum or regular unbounded $\kappa^*$-sum for some infinite regular cardinal $\kappa$. 
\end{defi}

\begin{defi} \label{shuffledef}
Suppose $\varphi$ is an order type. A \textit{$\varphi$-shuffle} (or \textit{shuffle over $\varphi$}) is an ordered sum of the form $\sum_{x \in \varphi} \psi_x$, where for each $x \in \varphi$ the set $\{y \in \varphi: \psi_x \leqslant \psi_y\}$ is dense in $\varphi$. 
\end{defi}

In \Cref{abstractthm} below, we will consider shuffles over types that are equimorphic with their own squares. We will need the following lemma and its corollary. 

\begin{lemm}\label{phi equi phi2 dense implies equi}
Suppose that $X$ is a linear order such that $X \equiv 2X$. Then any dense suborder $Y \subseteq X$ is equimorphic to $X$. 
\end{lemm}
\begin{proof}
It suffices to find a suborder $Y^*$ of $Y$ that is isomorphic to $X$. 

The hypothesis $X \equiv 2X$ guarantees that we can find a collection of $X$-many pairwise disjoint non-degenerate intervals in $X$. Explicitly, suppose $f: 2X \rightarrow X$ is an embedding. Writing $2X = \sum_{x \in X} 2$, let $I_x$ denote the image under $f$ of the $x$-th summand, and let $J_x$ denote the convex closure of $I_x$. Then $\{J_x: x \in X\}$ is a collection of pairwise disjoint non-degenerate intervals in $X$ whose order type in the induced order is $X$. By the density of $Y$ in $X$, for each $x \in X$ we can find $y_x \in J_x \cap Y$. Let $Y^* = \{y_x: x \in X\}$. Then $Y^* \subseteq Y$, and by construction $Y^* \cong X$.
\end{proof}

\begin{coro} \label{phi equi square dense implies equi}
Suppose that $X$ is a linear order such that $X \equiv X^2$. Then any dense suborder $Y \subseteq X$ is equimorphic to $X$. 
\end{coro}
\begin{proof}
The statement clearly holds when $X \cong 1$, so suppose that $X$ contains at least two points. Then $X \equiv X^2$ gives $X \equiv 2X$, and the result follows from \Cref{phi equi phi2 dense implies equi}.
\end{proof}

\begin{theo}
\label{abstract theorem} \label{abstractthm}
Suppose that $\mathcal{F}$ is a class of order types such that for every $\rho \in \mathcal{F}$, either $\rho \equim \rho^2$ or one of $\rho, \rho^*$ is an infinite regular cardinal.

Let $\mathcal{T}$ denote the class of order types obtained by closing $\{0, 1\}$ under $\rho$-shuffles (when $\rho \equiv \rho^2$) and regular unbounded $\rho$-sums (when one of $\rho, \rho^*$ is an infinite regular cardinal), for all $\rho \in \mathcal{F}$.

Then for every $\varphi \in \mathcal{T}$, either $\varphi = 1$ or $\varphi$ is equimorphic to a type $\tilde{\varphi}$ without finite $F$-classes.
\end{theo}

\begin{proof}

Fix $\varphi \in \mathcal{T} \setminus \{1\}$. We induct on the construction of $\varphi$ to show that $\varphi$ is equimorphic to a type without finite $F$-classes. 

If $\varphi = 0$ there is nothing to show. So assume $\varphi \neq 0$. Then either $\varphi$ is a regular unbounded $\rho$-sum or a $\rho$-shuffle, for some $\rho \in \mathcal{F}$. Suppose first that $\rho$ is an infinite regular cardinal and $\varphi$ is a regular unbounded $\rho$-sum of types $\psi_i \in \mathcal{T}$:
\begin{align}
\label{rus}
\varphi = \sum_{i \in \rho} \psi_i.
\end{align}

There are several possibilities to consider. First, it may be that for all large enough $i \in \rho$ we have $\psi_i = 0$. But then by the definition of regular unbounded sum, we actually have $\psi_i = 0$ for all $i \in \rho$. This gives $\varphi = 0$, contradicting our assumption. 

Second, it may be that for every $i \in \rho$ we have $\psi_i \leqslant 1$. Then by what we have just observed, it must be that for unboundedly-many $i \in \rho$ we have $\psi_i = 1$. By the regularity of $\rho$, it follows $\varphi = \rho$. Since an infinite cardinal has no finite $F$-classes, we are done in this case. 

Finally, it may be that for some $i \in \rho$ we have $\psi_i > 1$. Since $\varphi$ is a regular unbounded sum, it must be then that $\psi_i > 1$ for unboundedly-many $i \in \rho$. By induction, each such $\psi_i$ is equimorphic to some $\tau_i$ without finite $F$-classes. Define a new type:
\begin{align*}
\tilde{\varphi} = & \sum_{i \in \rho, \psi_i > 1} \tau_i.
\end{align*}

Clearly, $\tilde{\varphi}$ has no finite $F$-classes, and $\tilde{\varphi} \leqslant \varphi$. It remains to show that $\varphi \leqslant \tilde{\varphi}$. Using the regular unboundedness of the original sum $\varphi$, and the fact that $1$ embeds in \textit{every} summand $\tau_i$ in $\tilde{\varphi}$, we can recursively choose for each $i \in \rho$ some $i^* \in \rho$ such that $\psi_i \leqslant \tau_{i^*}$, and such that $i < j$ implies $i^* < j^*$. Once this is done, we may glue together the individual embeddings witnessing $\psi_i \leqslant \tau_{i^*}$ to get an embedding witnessing $\varphi \leqslant \tilde{\varphi}$, as desired. 

The argument when $\rho$ is instead the reverse of an infinite regular cardinal is symmetric. 

Now suppose that $\varphi$ is a $\rho$-shuffle of types $\psi_i \in \mathcal{T}$ for some $\rho \in \mathcal{F}$ with $\rho \equim \rho^2$:
\begin{align}
\label{shuffle}
\varphi = \sum_{i \in \rho} \psi_i.
\end{align}

Again, there are several possibilities to consider. As before, we cannot have $\psi_i = 0$ for all $i \in \rho$, since we are assuming $\varphi \neq 0$. If $\psi_i \leqslant 1$ for all $i$, and $\psi_i = 1$ for at least one $i$, then by the definition of $\rho$-shuffle we must have $\psi_i = 1$ for densely-many $i$. But then $\varphi = \rho'$ for a type $\rho'$ that is dense in $\rho$. By \Cref{phi equi square dense implies equi} we have $\rho' \equim \rho$, and so $\varphi \equim \rho$. If $\rho = 1$ (which, strictly speaking, we have not ruled out), then $\varphi = 1$, a contradiction since we are assuming $\varphi \neq 1$. Thus $\rho > 1$. Then since $\rho^2 \equim \rho$ it must be that $\rho$ is infinite. In particular we must have that at least one\footnote{Using that $\rho^2 \equiv \rho$, it is not hard to see that in fact both $\omega$ and $\omega^*$ embed in $\rho$.} of $\omega$ and $\omega^*$ embeds in $\rho$. Without loss of generality, assume $\omega \leqslant \rho$. Then $\omega \rho \leqslant \rho^2$. Hence $\omega \rho \leqslant \rho$, which yields $\omega \rho \equiv \rho \equiv \varphi$. Since $\omega \rho$ contains no finite $F$-classes, $\varphi$ is equimorphic to a type without finite $F$-classes in this case. 

The other possibility is that $\psi_i > 1$ for some $i \in \rho$. Then $\psi_i > 1$ for densely-many $i$, by definition of shuffle. By induction, each such $\psi_i$ is equimorphic to a type $\tau_i$ without finite $F$-classes. Letting $\rho' = \{i \in \rho: \psi_i > 1\}$, we have $\rho'$ is dense in $\rho$. Define a type $\tilde{\varphi}$:
\begin{align*}
\tilde{\varphi} = & \sum_{i \in \rho'} \tau_i.
\end{align*}
By construction, $\tilde{\varphi}$ has no finite $F$-classes. We clearly have $\tilde{\varphi} \leqslant \varphi$. We show $\varphi \leqslant \tilde{\varphi}$. 

Fix an order $X$ of type $\rho$ and a suborder $Y \subseteq X$ of type $\rho'$ that is dense in $X$. For each $i \in X$, fix an order $I_i$ of type $\psi_i$ so that $\sum_{i \in X} I_i$ has type $\varphi = \sum_{i \in \rho} \psi_i$. For each $i \in Y$, fix $J_i$ of type $\tau_i$ so that $\sum_{i \in Y} J_i$ has type $\tilde{\varphi} = \sum_{i \in \rho'} \tau_i$.

Since $X \equiv X^2$ we may (as in the proof of \Cref{phi equi phi2 dense implies equi}) find $X$-many pairwise disjoint intervals in $X$, i.e.\ a collection $\{Z_i: i \in X\}$ such that each $Z_i$ is a non-degenerate interval in $X$ and $i < i'$ implies $Z_i < Z_{i'}$. 

Fix $i \in X$. If $I_i > 1$, then since $Y$ is dense in $X$ and $\varphi$ is a $\rho$-shuffle, we can find $y_i \in Z_i \cap Y$ such that $I_i \leqslant J_{y_i}$. And if $I_i \cong 1$, for any fixed $y_i \in Z_i \cap Y$ we have $I_i \leqslant J_{y_i}$, since the types $\tau_i$ are non-empty. Thus in either case we can choose $y_i \in Z_i \cap Y$ such that $I_i \leqslant J_{y_i}$. If we fix embeddings $f_i: I_i \rightarrow J_{y_i}$, then the map $f: \sum_{i \in X} I_i \rightarrow \sum_{i \in Y} J_i$ defined by $f(i, x) = (y_i, f_i(x))$ is an embedding that witnesses $\varphi \leqslant \tilde{\varphi}$. Thus $\varphi \equiv \tilde{\varphi}$, as claimed. 
\end{proof}

\begin{coro}
\label{abstract corollary}
Let $\mathcal{T}$ be a family of order types as in \Cref{abstractthm}. Consider the family $\mathcal{S}$ consisting of all finite sums of members of $\mathcal{T}$. Then every member of $\mathcal{S}$ is equimorphic to an order type with only finitely many finite condensation classes.
\end{coro}

\subsubsection{Untranscendable $\sigma$-scattered types}\label{subsub sigma}

Laver showed that the class of indecomposable scattered linear orders has the form of a family $\mathcal{T}$ as in \Cref{abstract theorem}.

\begin{theo}[{Laver, \cite[Theorem 3.1(f) with $Q = \{1\}$]{973L1}}] \label{Laver sigma-scattered indec thm} 
Let $\mathcal{F}$ denote the family consisting of all infinite regular cardinals $\kappa$, their reverses $\kappa^*$, and all types of the form $\eta_{\alpha\beta}$. Then the class of indecomposable $\sigma$-scattered types is obtained by closing $\{0, 1\}$ under regular unbounded $\rho$-sums and $\rho$-shuffles, for $\rho \in \mathcal{F}$.  
\end{theo}

For the definition of the orders $\eta_{\alpha\beta}$, see \cite[pg. 99]{971L0}. We note that the indices $\alpha$ and $\beta$ here are regular uncountable cardinals. For our purposes, the relevant fact concerning these orders is that each is equimorphic to its square.

\begin{prop}[{Laver, \cite[Corollary 3.4]{971L0}}] \label{embedding in etas} Suppose $\varphi$ is an order type. Then $\varphi \leqslant \eta_{\alpha\beta}$ if and only if $\varphi$ is $\sigma$-scattered, $\alpha^* \isntcontainedin \varphi$, and $\beta \isntcontainedin \varphi$.
\end{prop}

\begin{coro}
\label{etas embed their squares}
$\eta_{\alpha\beta} \equim {\eta_{\alpha\beta}}^2$. 
\end{coro}

\begin{proof}
Since $\eta_{\alpha\beta}$ is $\sigma$-scattered, ${\eta_{\alpha\beta}}^2$ is $\sigma$-scattered as well. By \Cref{embedding in etas}, $\alpha^* \isntcontainedin \eta_{\alpha\beta}$. Since $\alpha$ is a regular cardinal we also have $\alpha^* \isntcontainedin {\eta_{\alpha\beta}}^2$. Likewise, the regularity of $\beta$ implies $\beta \isntcontainedin {\eta_{\alpha\beta}}^2$. Thus ${\eta_{\alpha\beta}}^2 \leqslant \eta_{\alpha\beta}$ by \Cref{embedding in etas}.
\end{proof}

We are now in a position to apply \Cref{abstract theorem} to $\sigma$-scattered linear orders.

\begin{theo}
\label{sigma has finite finite classes}
Every $\sigma$-scattered linear order is equimorphic to a linear order with only finitely many finite $F$-classes.
\end{theo}
\begin{proof}
By \cite[Theorem 3.1(d)]{973L1} every $\sigma$-scattered linear order is a finite sum of indecomposable $\sigma$-scattered linear orders.
By \Cref{Laver sigma-scattered indec thm}, \Cref{etas embed their squares}, and \Cref{abstract theorem}, every indecomposable $\sigma$-scattered linear order is either $1$ or equimorphic to a type without finite $F$-classes.
\end{proof}

\begin{theo} \label{SS+untr implies SI}
Every untranscendable $\sigma$-scattered linear order different from $2$ is strongly indecomposable.
\end{theo}

\begin{proof}
Suppose $\varphi \neq 2$ is $\sigma$-scattered and untranscendable. By \Cref{AD+UT=2}, $\varphi$ is indecomposable and hence belongs to the class $\mathcal{T}$ of indecomposable $\sigma$-scattered types. By \Cref{Laver sigma-scattered indec thm} and \Cref{etas embed their squares}, \Cref{abstract theorem} applies to $\mathcal{T}$; hence every type in $\mathcal{T}$ is equimorphic to a type without finite $F$-classes. In particular, $\varphi$ is equimorphic to such a type. By \Cref{2 only FF}, $\varphi$ is strongly indecomposable.  
\end{proof}

Since every countable order is $\sigma$-scattered, we obtain the following corollary. 

\begin{coro}
\label{2 only scattered}
Every countable untranscendable linear order different from $2$ is strongly indecomposable. \qed
\end{coro}

\Cref{SS+untr implies SI} says that for $\sigma$-scattered types $\varphi \neq 2$, untranscendability is a strengthening of strong indecomposability. We note that it is a strict strengthening, as, for example, $\omega^2$ is strongly indecomposable but transcendable. However, we lack a characterization of the scattered and $\sigma$-scattered untranscendable linear orders, cf.\ Problems \ref{characterise_ut&sc} and \ref{characterise_ut&s-sc}.

\subsubsection{Untranscendable Aronszajn types}\label{subsub_Aronszajn}

\begin{defi} \label{Aronszajnlinedef}
An \emph{Aronszajn line} is an uncountable linear order that does not embed either $\omega_1$ or $\omega_1^*$, or any uncountable suborder of $\mathbb{R}$. 
\end{defi}
Aronszajn lines were first constructed by Aronszajn \cite{935K0}, \cite[§5]{984T0}. For a history of this notion cf.\ \cite{995T0}. It can be shown that every Aronszajn line has cardinality $\aleph_1$.
The order type of an Aronszajn line is called an \textit{Aronszajn type}.  

The \textit{Proper Forcing Axiom} (\PFA) is a forcing axiom introduced by Baumgartner (cf.\ \cite{984B1}) that strongly influences the combinatorics of objects of size $\aleph_1$ (cf.\ \cite{010M2} for a survey). And indeed, it turns out that in the presence of \PFA\ the class of Aronszajn lines possesses a structure theory that resembles the structure theory of the class of $\sigma$-scattered linear orders; see the discussion below. We will use this structure theory to show that the untranscendable Aronszajn lines are strongly indecomposable. 

A \textit{universal Aronszajn type} is an Aronszajn type $\nu$ such that $\varphi \leqslant \nu$ for every Aronszajn type $\varphi$. Observe that if there is a universal Aronszajn type, then it is unique up to equimorphism. Moore has shown that the Proper Forcing Axiom implies the existence of such a type, cf. \cite{009M0}:

\begin{theo}[$\ZFC + \PFA$]
\label{universal Aronszajn}
There is a universal Aronszajn type $\nu$. 
\end{theo}

\begin{defi}[$\ZFC + \PFA$] \label{fragmentedAlinedef}
An Aronszajn type $\varphi$ is called \emph{fragmented} if $\nu \isntcontainedin \varphi$.
\end{defi}

The universal Aronszajn type $\nu$ can be viewed as an analog\ue\ of the rational type $\eta$, which is universal for countable types. From this point of view, a fragmented Aronszajn type is the analog\ue\ of a scattered type. 

\textit{Countryman lines} are a special type of Aronszajn line that play an important role in the structure theory of Aronszajn lines. 

\begin{defi} \label{Countrymandef}
A linear order $C$ is called a \textit{Countryman line} if 
\begin{itemize}
    \item[(i)] $C$ is uncountable, 
    \item[(ii)] Its square $C \times C$, when partially ordered coordinatewise, can be expressed as a union of countably many linearly ordered subsets. 
\end{itemize}
\end{defi}

It can be shown that a Countryman line is necessarily Aronszajn. Notice that if $C$ is Countryman, then so is $C^*$. 

Countryman lines were first shown to exist by Shelah in response to a question of Countryman, \cite{976S0}. Moore showed that, under PFA, every Countryman line $C$ is minimal among the Aronszajn lines in the sense that every Aronszajn line $A$ embeds either $C$ or its reverse $C^*$, see \cite{006M1}.

For the remainder of this subsection we fix a Countryman line $C$. Let $\gamma$ denote the order type of $C$ and $\gamma^*$ the order type of its reverse $C^*$. 

%We will use the following fact in our proof that, under $\mathsf{PFA}$, every untranscendable Aronszajn line is strongly indecomposable. 

We will use the following result of Martinez-Ranero, cf. \cite[Fact 3.11]{011M2}, in our proof of \Cref{Aron+untr implies SI}:

\begin{prop}[$\ZFC + \PFA$]
\label{Countryman lines equi to squares}
If $\rho$ is a Countryman type then $\rho \equiv \rho^2$.
\end{prop}

The following definitions and results, due to Barbosa (cf.\ \cite{023B0}), spell out the analogy between the class of $\sigma$-scattered linear orders and the class of fragmented Aronszajn types under $\mathsf{PFA}$. 

\begin{defi}[{\cite[Definition 3.13 with $Q = \{1\}$]{023B0}}]
\label{definition of H}
Let $\mathcal{H}$ denote the class of linear order types obtained by closing $\{0, 1\}$ under $\varphi$-regular unbounded sums, for $\varphi \in \{\omega, \omega^*\}$, and under $\varphi$-shuffles, for $\varphi \in \{\eta, \gamma, \gamma^*\}$. 
\end{defi}

Recall that $\mathcal{C}$ denotes the class consisting all fragmented Aronszajn types along with all countable types. 
By Definition 3.11 and Lemma 3.12 from \cite{023B0}, we have $\mathcal{H} \subseteq \mathcal{C}$.

Since under $\mathsf{PFA}$ we have $\gamma \equiv \gamma^2$ by \Cref{Countryman lines equi to squares}, under $\mathsf{PFA}$ the class $\mathcal{H}$ satisfies the conclusion of \Cref{abstract theorem}. In light of \Cref{Laver sigma-scattered indec thm}, $\mathcal{H}$ can be viewed as an analog\ue\ of the class of indecomposable $\sigma$-scattered orders. As Barbosa's following results show, the analogy is quite strong. For the first result, cf. \cite[Proposition 3.14]{023B0}:

\begin{prop}[$\ZFC + \PFA$]
\label{HcAI}
Every member of $\mathcal{H}$ is indecomposable.
\end{prop}

Moreover, Barbosa has shown that, just as every $\sigma$-scattered linear order can be written as a finite sum of indecomposable orders, we have the following, \cite[Theorem 3.19 with $Q = \{1\}$]{023B0}:

\begin{theo}[$\ZFC + \PFA$]
\label{finite sums of AI Aronszajns}
Every type $\varphi \in \mathcal{C}$ can be written as a finite sum of members of $\mathcal{H}$.
\end{theo}

The following result is the analog\ue\ of \Cref{Laver sigma-scattered indec thm} for Aronszajn types under \PFA. 

\begin{theo}[$\ZFC + \PFA$] \label{Barbosa indec theorem}
The indecomposable members of $\mathcal{C}$ are exactly the types equimorphic to an element of $\mathcal{H}$.
\end{theo}

\begin{proof}
By \Cref{HcAI}, every type in $\mathcal{H}$ is indecomposable and thus also every type equimorphic to an element of $\mathcal{H}$ is indecomposable. 

Conversely, by \Cref{finite sums of AI Aronszajns}, every type $\varphi \in \mathcal{C}$ can be written as a finite sum of elements of $\mathcal{H}$. Thus if $\varphi$ is indecomposable, it must be equimorphic with one of these finitely many summands.
\end{proof}

This result is interesting because it ties in nicely with another fact about $\mathcal{H}$:

\begin{theo}[$\ZFC + \PFA$]
Every member of $\mathcal{H}$ except $1$ is equimorphic to a type without finite $F$-classes.
\end{theo}

\begin{proof}
As $\varphi^2 \equim \varphi$ for all three $\varphi \in \{\eta, \gamma, \gamma^*\}$ and $\omega$ is an infinite regular cardinal, the class $\mathcal{F} = \{\omega, \omega^*, \eta, \gamma, \gamma^*\}$ satisfies the assumptions of \Cref{abstract theorem}. By \Cref{definition of H}, its conclusion provides what was demanded.
\end{proof}

Via \Cref{finite sums of AI Aronszajns} and \Cref{Barbosa indec theorem} this yields an analog\ue\ of \Cref{sigma has finite finite classes}:

\begin{coro}[$\ZFC + \PFA$]
\label{Aronszajn has finite finite classes}
Every member of $\mathcal{C}$ is equimorphic to a linear order with only finitely many finite $F$-classes.
\end{coro}

Via \Cref{2 only FF} we arrive at the following:

\begin{coro}[$\ZFC + \PFA$]
Every untranscendable member of $\mathcal{C}$ except $2$ is strongly indecomposable.
\end{coro}

%Taken together, \Cref{Countryman lines equi to squares} and Theorems \ref{Barbosa indec theorem} and \ref{abstract theorem} give that under \PFA\ any indecomposable type $\varphi$ that is either countable or fragmented Aronszajn is equimorphic to a type without finite $F$-classes. Hence if $\varphi \neq 2$ is untranscendable and either countable or fragmented Aronszajn then it is strongly indecomposable, by \Cref{2 only FF}. 

By its universality (and the fact that the product of two Aronszajn types is Aronszajn), any universal Aronszajn type is equimorphic to its square, and hence untranscendable. Thus to conclude that \emph{every} untranscendable Aronszajn type is strongly indecomposable under \PFA, it remains to show that the universal Aronszajn type $\nu$ is strongly indecomposable. By \Cref{2 only FF}, it suffices to check that $\nu$ is equimorphic to a type without finite $F$-classes. But, again by its universality, $\nu$ is equimorphic to $\omega \nu$, which has no finite $F$-classes. We have proved the following. 

\begin{theo}[$\ZFC + \PFA$]
\label{Aron+untr implies SI}
Every untranscendable Aronszajn line is strongly indecomposable.
\end{theo}

\subsection{The reals}

Some results about the reals do not depend on the axiom of choice but others do. (As usual, in what follows $\AC$ denotes the axiom of choice, cf.\ \cite{904Z0}, $\ZF$ denotes Zermelo-Fraenkel set theory without the axiom of choice and $\ZFC$ is an abbreviation for $\ZF + \AC$.) For example the following corollary of \Cref{homogeneity =>s-untranscendability} was proved without \AC:

\begin{coro}[$\ZF$]
\label{R is untranscendable}
$\otR$ is untranscendable.
\end{coro}

In contrast, this result due to Sierpi\'nski, cf.\ \cite[Theorem 9]{956ER0} requires \AC:

\begin{theo}[$\ZFC$]
\label{fullsizerealtype=>-SI}
No order type $\varphi \leqslant \otR$ of cardinality continuum is strongly indecomposable.
\end{theo}

In order to appreciate how different the situation may look in the absence of the axiom of choice let us first mention 
$\AD$---the axiom of determinacy, cf.\ \cite{962MS0,020C0, 021L0} as it paints a markedly different image of the reals. It implies the axiom of choice for countable families of reals, cf.\ \cite{964MS0}, but is incompatible with the full axiom of choice, cf.\ \cite{953GS0}. It has considerable consistency strength, cf.\ \cite[\S 32]{003K0}, and implies, that all sets of reals enjoy regularity properties such as being Lebesgue-measurable, cf.\ \cite{964MS0}, having the property of Baire, cf.\ \cite{957O0}, but also, cf.\ \cite{964D0}, that
\begin{align}
\label{PSP}
\text{Every uncountable set of reals contains a non-empty perfect subset.}
\end{align}
Specker had realised early on that \eqref{PSP} implies $\aleph_1$ to be a limit cardinal in $L$, cf.\ \cite{957S1}. Solovay later proved the consistency of \eqref{PSP} assuming the existence of an inaccessible cardinal, cf.\ \cite{970S0} (and $\aleph_1$'s regularity is necessary for that, cf.\ \cite{974T3}).

Recall that a \emph{Bernstein set} is a set of reals that intersects every perfect set but does not contain any perfect set. The existence of such a set can be proved with the axiom of choice, cf.\ \cite{908B0}, and clearly can be refuted using \eqref{PSP}. But as no Bernstein set can have the property of Baire, cf.\ \cite[Theorem 5.4]{980O0}, the nonexistence of a Bernstein set has no consistency strength beyond $\ZFC$, cf.\ \cite{984S0}, even in conjunction with $\aleph_1$ being regular.

Bernstein sets are relevant to our discussion due to the following result:

\begin{prop}[$\ZF$]
\label{Bernstein}
$\otR$ is strongly indecomposable if and only if there is no Bernstein set.
\end{prop}

\begin{proof}
First suppose $\RR = A \cup B$ is a partition of $\RR$ into two suborders neither of which embed an order-isomorphic copy of $\RR$. Then since every perfect set contains a copy of $\RR$, neither $A$ nor $B$ contains a perfect set. Hence $A$ and $B$ are both Bernstein sets. 

Conversely, suppose $A$ is a Bernstein set. Then $B = \RR\setminus A$ is also a Bernstein set. We claim neither $A$ nor $B$ contains an isomorphic copy of $\RR$. This follows if we can show that every isomorphic copy of $\RR$ contains a perfect set. 

Suppose $R$ is a suborder of $\RR$ isomorphic to $\RR$. While $R$ of course contains a subset which is perfect with respect to itself (namely, itself), because $R$ need not be closed in $\RR$ we need to look a little harder to find an $\RR$-perfect set in $R$. 

Observe that $R$ contains all but countably many of its limit points, i.e.\ there is a countable $C \subseteq \RR$ s.t.\ if $x \in \RR\setminus C$ is a limit of points $x_n \in R$, then $x \in R$. (One way to see this: let $f: \RR \longrightarrow R$ be an isomorphism. Then $f$ is increasing and hence discontinuous at only countably many points.) Let $S$ be the closure of $R$. Then $S$ is perfect (i.e.\ perfect in $\RR$). To find a perfect subset of $S$ that avoids $C$ observe that $S \setminus C$ is an uncountable $G_\delta$-set. Therefore $S\setminus C$ contains a non-empty perfect subset, cf.\ \cite{916A0,916H0}.
%(This is always possible since $C$ is countable: e.g. enumerate $C = c_1, c_2, ...$ and iteratively choose a pairwise disjoint collection of open intervals\todo[color=green!40]{I do not see neither why this works (what if $c_1$ is a limit point of $C$?) nor why it produces the required perfect set ($\bigcup_{n < \omega} I_n$ could contain $S$)} $\seq{I_n}{n < \omega}$ with $c_n \in I_n$. Then $S \setminus \bigcup_{n < \omega} I_n$ is perfect). 
\end{proof}

L\"ucke, Schlicht, and Weinert could show that if all sets of reals have the property of Baire, even stronger partition properties follow:

\begin{theo}[$\ZF$]\label{theorem_ADlambda}
If all sets of reals have the property of Baire, then $\otR \longrightarrow (\otR)^2_n$ holds for all natural numbers $n$.
\end{theo}

This follows from \cite[Theorem 2.1.2]{
017LSW0}. The result is proved there for the Cantor space instead of $\otR$. Since these two linear order types are equimorphic, the result in the paper and the one just mentioned are actually equivalent, cf.\ \Cref{monotonicity}.

\subsubsection{Consequences of the Proper Forcing Axiom}

%Now we are going to mention some consequences of

Recall that for a given cardinal $\kappa$, an ordered set $X$ is called \emph{$\kappa$-dense}, if every nonempty interval of it as well as $X$ itself has cardinality $\kappa$.

One of the earliest applications of the \emph{Proper Forcing Axiom}, $\PFA$ was Baumgartner's proof, cf. \cite{973B0}, that consistently for any given infinite cardinal $\kappa < 2^{\aleph_0}$, all $\kappa$-dense sets of reals are isomorphic.

\begin{theo}[$\ZFC + \PFA$]
\label{Baumgartner}
All $\aleph_1$-dense sets of reals are isomorphic.
\end{theo}

The following fact is folklore, but we could not locate a proof in the literature.

\begin{coro}[$\ZFC + \PFA$]
All real types of cardinality $\aleph_1$ are mutually equimorphic.
\end{coro}

\begin{proof}
In light of \Cref{Baumgartner}, it is sufficient to show that every set $X \subseteq \mathbb{R}$ of cardinality $\aleph_1$ contains and is contained in an $\aleph_1$-dense set of reals.

If $Y$ is any $\aleph_1$-dense set of reals that is also $\aleph_1$-dense in $\mathbb{R}$, then $X \cup Y$ is $\aleph_1$-dense and contains $X$. 

In the other direction, define a relation $\sim$ on $X$ by the rule $x \sim y$ if the interval $[\{x, y\}]$ is countable. It is routine to check that $\sim$ is a condensation of $X$. Since $\mathbb{R}$ is separable, every $\sim$-class must be countable, and it follows that $X/{\sim}$ is $\aleph_1$-dense. Let $X'$ be obtained by choosing one point from every $\sim$-class. Then $X'$ is an $\aleph_1$-dense suborder of $X$, as desired. 
\end{proof}

\begin{coro}[$\ZFC + \PFA$]
\label{PFA=>indivisible real types}
All real types of cardinality $\aleph_1$ are strongly indecomposable, in fact, we even have $\rho \longrightarrow (\rho)^1_{\aleph_0}$ for every real type $\rho$ of cardinality $\aleph_1$.
\end{coro}

\begin{proof}
Let $R \subset \RR$ be a set of order type $\rho$ of $\aleph_1$ real numbers and let $\chi : R \longrightarrow \omega$ be a col\ou ring of $R$. Then one col\ou r class has to be uncountable and, by the previous corollary, this class is equimorphic to $R$ so in particular contains an order-isomorphic copy of $R$.
\end{proof}

\begin{theo}[$\ZFC + \PFA$]
All untranscendable real types of cardinality at most $\aleph_1$ are either strongly indecomposable or equal to $2$.
\end{theo}

\begin{proof}
$\PFA$ implies that there are $\aleph_2$ real numbers and by \Cref{2 only scattered} all countable untranscendable types are strongly indecomposable or equal to $2$. So it remains to show that untranscendable linear orders of cardinality $\aleph_1$ are strongly indecomposable. As all real types of cardinality $\aleph_1$ are strongly indecomposable by \Cref{PFA=>indivisible real types} we are done.
\end{proof}

\section{Open Problems\dots} \label{OpenProbs}

\subsection{\dots with choice}
We remarked previously that we do not know if the hypothesis of $s$-untranscendability in the statement of \Cref{Garrett's Theorem} can be weakened to untranscendability. 

\begin{prob}
\label{2x=x=x2 =?=> x=xx}
Is there an untranscendable order type $\varphi$ such that $\varphi \equim \varphi 2 \equim 2 \varphi$ but $\varphi \not \equim \varphi^2$?
\end{prob}

In light of \Cref{fullsizerealtype=>-SI} one may wonder whether uncountable real types are (at least consistently)  the only untranscendable types different from $2$ which are not strongly indecomposable. %all untranscendable Baumgartner types are consistently strongly indecomposable or maybe even provably strongly indecomposable.
Being ever so slightly partial to a positive answer, we pose the following conjecture:

\begin{conj}[The Trichotomy Conjecture (possibly assuming \PFA)] \label{realtypeconj}
For every untranscendable linear order $\varphi$ at least one of the following three statements applies:
\begin{itemize}
\item $\varphi$ is strongly indecomposable,
\item $\varphi$ contains an uncountable real type,
\item $\varphi = 2$.
\end{itemize}
\end{conj}

What would be very helpful in proving this conjecture would be a better understanding of \emph{Baumgartner types}. These order types are neither $\sigma$-scattered, nor do they contain an uncountable real type, nor an Aronszajn type. They are named after James Baumgartner who described them in \cite{976B1}. There have been further investigations since, cf.\ \cite{009IM0,018LM0,019L0,022L0, 024CEM0, 024S0}. At present, however, there is no structure theorem for the indecomposable Baumgartner types à la \Cref{Laver sigma-scattered indec thm}  (for $\sigma$-scattered types) and \Cref{{Barbosa indec theorem}} (for Aronszajn types) that would allow for the application of the methods of this paper.

The following question is strongly related to \Cref{realtypeconj}. 

\begin{ques}
    Does $\ZFC$ prove that all untranscendable Aronszajn lines are strongly indecomposable?
\end{ques}

We would like to pose more open-ended problems:

\begin{prob}
\label{characterise_ut&sc}
    Characteri\z e the class of all %infinite
    untranscendable scattered linear orders.
\end{prob}

\begin{prob}
\label{characterise_ut&s-sc}
    Characteri\z e the class of all %infinite
    untranscendable $\sigma$-scattered linear orders.
\end{prob}

Also note that the notion of multiplicative decomposability, cf.\ \Cref{multiplicative_decomposability}, makes perfect sense for linear order types---call $\rho$ decomposable if there are types $\varphi, \psi < \rho$ such that $\rho = \varphi\psi$. So we may also pose the following:

\begin{prob}
    Characteri\z e the class of all infinite multiplicatively indecomposable scattered linear orders.
\end{prob}

\subsection{\dots and without}\label{sub_noAC}

Many of the arguments in this paper do not depend on the axiom of choice, but some do, including the proof of \Cref{product-lemma}. We wonder whether this can be avoided:

\begin{ques}\label{ques:AC}
Is it provable with $\ZF$ that if $\rho\tau \leqslant \varphi\psi$, then necessarily $\rho \leqslant \varphi$ or $\tau \leqslant \psi$?
\end{ques}

Notice that $\ZF$ suffices to prove that $\rho\tau \leqslant \varphi\psi$ implies $\rho \leqslant \varphi 2$ or $\tau \leqslant \psi$.

%\todo[inline, color = thilo, caption = {}]{We know that the proof above works if $X$ or $W$ can be well-ordered or if $\AC_{\card{Y}}$ holds true. Moreover we can prove with $\ZF$ that
%\begin{enumerate}
%\item If $\rho$ has a left endpoint and $\rho\tau \leqslant \varphi\psi$, then $\rho \leqslant \varphi$ or $\tau \leqslant \psi$,
%\item If $\rho$ has a right endpoint and $\rho\tau \leqslant \varphi\psi$, then $\rho \leqslant \varphi$ or $\tau \leqslant \psi$,
%\item $\rho\tau \leqslant \varphi\psi$ implies $\rho \leqslant \varphi 2$ or $\tau \leqslant \psi$,
%\end{enumerate}
%Can we show anything else with $\ZF$, maybe assuming that $\varphi \leqslant \rho$ and $\psi \leqslant \tau$?
%
%Can we find a counterexample? Maybe under the assumption that the reals are a countable union of countable sets\dots
%
%\dots a counterexample under a determinacy regime would also be interesting.}

We currently only have a few examples of untranscendable types failing to be strongly indecomposable. 
Moreover, with the exception of $2$, our examples require the Axiom of Choice for the proof.
It is tempting to conjecture that this is no coincidence. Let us abbreviate the statement that $2$ is the only untranscendable linear order $\varphi$ failing to be strongly indecomposable as $\BE$ (\emph{binary exceptionalism}).

\begin{ques}
\label{ques : Only Two}
Is $\ZF + \BE$ consistent?
\end{ques}

Let us recall that by the conjunction of \Cref{R is untranscendable} and \Cref{fullsizerealtype=>-SI} , $\BE$ refutes the Axiom of Choice. An important statement in this context is the \emph{ordering principle}, abbreviated as $\OP$. It claims that every set can ordered linearly. It is known that $\ZF$ fails to imply it, cf.\ \cite[Chapter 7 \S 3]{973J0}, and that $\ZF + \OP$ fails to imply $\AC$, cf.\ \cite[Chapter 7 \S 2]{973J0}. An affirmative answer to \Cref{ques : Only Two} may be given with the help of a model in which comparatively few sets can be ordered linearly. Therefore the following variation of the question is also of interest.

\begin{ques}
Is $\ZF + \BE + \OP$ consistent?
\end{ques}

Often when a statement contradicts the axiom of choice, one can at least obtain its consistency with the \emph{axiom of dependent choice}, $\DC$, cf.\ \cite{942B0}. It is known that the axiom of choice for countable families does not suffice to yield $\DC$, cf.\ \cite[Chapter 8, \S 2]{973J0}, but that %it is a theorem in
$\ZF + \AD + V = L(\RR)$ proves $ \DC$, cf.\ \cite{984K0}. This leads us to yet another question:

\begin{ques}
Is $\ZF + \DC + \BE + \OP$ consistent?
\end{ques}

It should be pointed out that determinacy models will not help answer either of the two last questions. This is because, the following holds:

\begin{theo}[$\ZF+\AD$]\label{R/Q_is_unorderable}
$\soafft{\omega}{2}/E_0$ cannot be ordered linearly.
\end{theo}

This has been known for some time now via various routes\footnote{See for example the discussion at https://mathoverflow.net/questions/26861/explicit-
ordering-on-set-with-larger-cardinality-than-r and especially Caicedo’s contribution, cf. \cite{010C2}.}.

In fact, in the appendix we are going to provide a relatively direct proof of the slightly stronger \Cref{E_0-Theorem}.

Moreover, in \cite[page 163]{006W0} it is pointed out that $\ZF + \DC + \AD_\RR$ proves the existence of $X_0$ such that $\card{X_0} + \card{X_0} > \card{X_0}$ ($\AD_\RR$ is the strengthening of \AD\ where the players play reals).
So in a determinacy context there are infinite sets which are, so to speak, decomposable with respect to injections simpliciter, even before considering any order relation on them.

Note that by \Cref{R is untranscendable}, $\BE$ implies that $\otR$ is strongly indecomposable. 
The latter is the case if and only if there is a Bernstein set, cf.\ \Cref{Bernstein}. 
Therefore the following question could be of interest too:

\begin{ques}
Does $\ZF + \DC + \OP$ imply the existence of a Bernstein set?
\end{ques}

Although fragments of the axiom of choice and their mutual implications have been widely investigated, cf.\ e.g.\ \cite{973J0,998HR0,006H0}, apparently this question has not been answered yet.

Finally, our proof of \Cref{2 only FF} relies on the Axiom of Choice. One wonders whether there is a proof which does not do so. Therefore we are pondering the following question:

\begin{ques}
Does $\ZF$ imply that $2$ is the sole untranscendable yet not strongly indecomposable order type $\varphi$ with only finitely many finite $F$-classes? Does $\ZF + \DC$ imply it? $\ZF + \OP$? $\ZF + \DC + \OP$? What happens in a determinacy regime?
\end{ques}

\appendix
\section{On linear unorderability}

In this section we prove \Cref{E_0-Theorem}, which is a strengthening of \Cref{R/Q_is_unorderable}. We consider the following game $G(s,E,<,i)$ where $s \in \soafft{<\omega}{2}$, $E$ is an equivalence relation on $\soafft{\omega}{2}$, $<$ is a linear order on $\soafft{\omega}{2} / E$, and $i \in \{1,2\}$ slect a player. This is a general Banach-Mazur-game, cf.\ \cite{957O0}, \cite[Chapter 6]{003K0}, \cite[page 116]{015M0}.
Both players are taking turns playing finite non-empty sequences of zeros and ones thereby creating an infinite sequence $x$ through concatenation.  Let
\begin{align*}
f : \soafft{\leqslant\omega}{2} & \longrightarrow \soafft{\leqslant\omega}{2}\\
s & \longmapsto (j\mapsto 1-s(j)).
\end{align*}
be the flip, i.e.\ the function replacing an finite or infinite sequence $s$ by the sequence having zeros at the positions where $s$ has ones and vice versa. %We consider\todo[color=green!40]{
$G(s,E,<,i)$ is %}
the game in which Player $i$ wins iff $[s\conc f(x)]_E < [s\conc x]_E$.
%Note that who wins the game only depends on the $E$-equivalence class of $s\conc x$.

\begin{lemm}[$\ZF$]
\label{game-lemma}
If $E$ is an equivalence relation on $\soafft{\omega}{2}$ and $s \in \soafft{<\omega}{2}$ is such that for no $x \in \soafft{\omega}{2}$ we have $s\conc x \,E\, s\conc f(x)$, then Player $2$ fails to have a winning strategy in $G(s,E,<,i)$.
\end{lemm}

\begin{proof}
    Assume towards a contradiction that Player $2$ has a winning strategy $\rho$ in some game $G(s,E,<,i)$ where $s$, $E$, $<$, and $i$ are as above. Since $[s\conc f(x)]_E \ne [s\conc x]_E$ for all $x \in \soafft{\omega}{2}$ and $<$ is linear, Player $3-i$ wins iff $[s\conc x]_E < [s\conc f(x)]_E$.
    We define a strategy $\sigma$ for Player $1$ by letting $\sigma (\langle\rangle) = \langle 1\rangle \conc f(\rho(\langle 0\rangle))$ and, for sequences of moves of positive even length $\langle s_0, \dots, s_{2k-1} \rangle$ with length of $s_0 = \langle j \rangle \conc s'_0$ larger than one (by definition of $\sigma (\langle\rangle)$ it suffices to take care of these sequences) $\sigma (\langle s_0, \dots, s_{2k-1} \rangle) = f(\rho(\langle f(\langle j \rangle), f(s'_0), f(s_1), \dots, f(s_{2k-1})\rangle))$.
    
We now let both Players fight it out using their respective strategies $\sigma$ and $\rho$. That is,
\begin{align*}
    s_{2k} & = \sigma(\seq{k<\omega}{k<2k}),\\
    s_{2k+1} & = \rho(\seq{k<\omega}{k\leqslant 2k}),\\
    x & = s_0 \conc s_1\conc\dots
\end{align*}
As $\rho$ is winning, we have $[s\conc x]_E < [s\conc f(x)]_E$ if $i = 1$ and $[s\conc x]_E > [s\conc f(x)]_E$ if $i=2$.

Now note that $s_0 = \langle 1\rangle \conc f(\rho(\langle 0\rangle))$ and when $k>0$ we obtain $s_{2k}$ by applying $f \circ \rho$ to the sequence consisting of 
\begin{multline*}
\langle 0 \rangle, \rho(\langle 0\rangle), f(s_1), \rho (\langle \langle 0 \rangle, \rho(\langle 0\rangle), f(s_1) \rangle), f(s_3), \dots,\\
f(s_{2k-1}), \rho (\langle \langle 0 \rangle, \rho(\langle 0\rangle), f(s_1), \dots, f(s_{2k-1}) \rangle).
\end{multline*}
This means that $f(x)$ is the result of the play where Player $1$ starts by playing $\langle 0\rangle$ and then plays $f(s_{2k-1})$ in round $2k$ with $k>0$, while Player $2$ plays according to $\rho$.

But then, as $f$ is an involution and $\rho$ is a winning strategy for Player $2$ in $G(s,E,<,i)$, if $i = 1$ we have $[s\conc f(x)]_E < [s\conc x]_E$, while if $i=2$ then $[s\conc x]_E < [s\conc f(x)]_E$, a contradiction!
\end{proof}

\begin{theo}[$\ZF+\AD$]
\label{E_0-Theorem}
If $E \supseteq E_0$ is an equivalence relation on $\soafft{\omega}{2}$ such that for no $s \in \soafft{<\omega}{2}$ and $x \in \soafft{\omega}{2}$ we have $s\conc x \,E\, s\conc f(x)$, then $\soafft{\omega}{2} / E$ cannot be ordered linearly.
\end{theo}

\begin{proof}
Assume towards a contradiction that the axiom of determinacy holds true and $<$ is a linear order of $\soafft{\omega}{2} / E$. By \Cref{game-lemma} Player $2$ does not have a winning strategy in $G(\langle\rangle,E,<,1)$. So by $\AD$, Player $1$ has a strategy $\rho$ which allows them to win any play of $G(\langle\rangle,E,<,1)$. Let $s' = \rho(\langle\rangle)$. Consider the strategy for Player $2$ defined by
\begin{align*}
    \sigma (\seq{s_k}{k<2i+1}) = \rho(\langle s'\rangle\conc\seq{s_k}{k<2i+1}).
\end{align*}
%If $1$ and $2$ play according to $\rho$ and $\sigma$ they produce a sequence $\seq{t_k}{k < \omega}$ of moves such that
By \Cref{game-lemma} $\sigma$ cannot be a winning strategy in $G(s',E,<,2)$ so there is a sequence $\seq{t_{2k}}{k < \omega}$ of moves of Player $1$ such that, if Player $2$ adheres to $\sigma$, Player $1$ wins the play of $G(s',E,<,2)$. This means that 
$[s'\conc y]_E < [s'\conc f(y)]_E$ where
\begin{align*}
y & = t_0\conc t_1\conc\dots\text{ and}\\
t_{2k+1} & = \sigma(\seq{t_m}{m \leqslant 2k}).
\end{align*}
As $s'\conc y$ is also the result of a play of $G(\langle\rangle,E,<,1)$ in which Player $1$ adhered to their winning strategy $\rho$, we have $[f(s'\conc y)]_E < [s'\conc y)]_E$ and hence $[f(s'\conc y)]_E < [s'\conc f(y)]_E$. Now, as $E \supseteq E_0$, we have $[s'\conc f(y)]_E = [f(s')\conc f(y)]_E = [f(s'\conc y)]_E$, a contradiction!
\end{proof}

\Cref{R/Q_is_unorderable} now follows immediately.

\providecommand{\bysame}{\leavevmode\hbox to3em{\hrulefill}\thinspace}
\providecommand{\MR}{\relax\ifhmode\unskip\space\fi MR }
% \MRhref is called by the amsart/book/proc definition of \MR.
\providecommand{\MRhref}[2]{%
  \href{http://www.ams.org/mathscinet-getitem?mr=#1}{#2}
}
\providecommand{\href}[2]{#2}

\end{document}